\documentclass[11pt]{article}

\usepackage[dvipsnames]{xcolor}

\usepackage[margin={2.5cm,2.2cm}]{geometry} 

\usepackage{graphicx}
\usepackage{float}
\usepackage{subcaption}

\usepackage{tikz}

\usepackage{enumerate}
\usepackage{amsmath,amsthm,amsfonts,bm}
\usepackage{multirow}

\usepackage[nocompress]{cite}

\theoremstyle{plain}
\newtheorem{theorem}{Theorem}[section]
\newtheorem{proposition}[theorem]{Proposition}

\newtheorem{corollary}[theorem]{Corollary}
\newtheorem{remark}[theorem]{Remark}

\theoremstyle{definition}

\newtheorem{hyp}{Hypothesis}

\usepackage{mathtools}
\newcommand*{\defeq}{\mathrel{\vcenter{\baselineskip0.5ex \lineskiplimit0pt
			\hbox{\scriptsize.}\hbox{\scriptsize.}}}%
	=}
\newcommand{\R}{\mathbb{R}}

\newcommand{\bv}{\overline{v}}
\newcommand{\bu}{\overline{u}}
\newcommand{\bmu}{\overline{\mu}}
\newcommand{\bn}{\overline{n}}

\newcommand{\nn}{\mathbf{n}}

\newcommand{\vK}{v_{K}}

\newcommand{\vL}{v_{L}}

\newcommand{\uKs}{u_{K^*}}
\newcommand{\uL}{u_{L}}
\newcommand{\up}{\Pi_1^h u}

\newcommand{\nKs}{n_{K^*}}

\newcommand{\muK}{\mu_{K}}
\newcommand{\muL}{\mu_{L}}
\newcommand{\mup}{\Pi_0\mu}

\newcommand{\nabland}{\nabla_{\nn_e}^0}

\usepackage{hyperref}

\def\escalar#1#2{\left(#1,#2\right)}

\def\escalarL#1#2{\escalar{#1}{#2}}
\def\escalarLd#1#2{\escalar{#1}{#2}}
\def\escalarML#1#2{\escalar{#1}{#2}_h}
\def\dualH#1#2{\langle#1,#2\rangle}

\def\norma#1{\left\|#1\right\|}
\def\normaL#1{\norma{#1}_{L^2(\Omega)}}

\def\salto#1{\left[\!\left[#1\right]\!\right]}
\def\media#1{\left\{\!\!\left\{#1\right\}\!\!\right\}}

\def\T{\mathcal{T}}
\def\E{\mathcal{E}}

\def\N{\mathbb{N}}
\def\P{\mathbb{P}}
\def\X{\mathcal{X}}

\def\Ehi{\mathcal{E}_h^{\text{i}}}

\def\Pd{\mathbb{P}^{\text{disc}}}
\def\Pc{\mathbb{P}^{\text{cont}}}

\def\aupw#1#2#3{a_h^{\text{upw}}(#1;#2,#3)}

\title{\textbf{
A structure-preserving upwind DG scheme for a degenerate phase-field tumor model}}

\makeatletter
\def\@fnsymbol#1{\ensuremath{\ifcase#1\or \dagger\or \ddagger\or \mathsection\or * \or\mathparagraph\or *\or **\or \dagger\dagger \or \ddagger\ddagger \else\@ctrerr\fi}}
\makeatother

\author{Daniel Acosta-Soba\thanks{Departamento de Matemáticas, Universidad de Cádiz, Puerto Real, 11510 Cádiz, Spain -- Email: \texttt{daniel.acosta@uca.es}} \thanks{Department of Mathematics, University of Tennessee at Chattanooga, Chattanooga, TN 37403, USA}~,
~Francisco Guillén-González\thanks{Departamento de Ecuaciones Diferenciales y Análisis Numérico \& IMUS, Universidad de Sevilla, 41012 Seville, Spain -- Email: \texttt{guillen@us.es}}~,
~J. Rafael Rodríguez-Galván\thanks{Departamento de Matemáticas, Universidad de Cádiz, Puerto Real, 11510 Cádiz, Spain -- Email: \texttt{rafael.rodriguez@uca.es} -- Corresponding author}}

\begin{document}

\bibliographystyle{abbrv} 

\maketitle

\begin{abstract}

In this work, we present a modification of the phase-field tumor growth model given in \cite{hawkins-daarud_numerical_2012} that leads to bounded, more physically meaningful, volume fraction variables. In addition, we develop an upwind discontinuous Galerkin (DG) scheme preserving the mass conservation, pointwise bounds and energy stability of the continuous model. Finally, some computational tests in accordance with the theoretical results are introduced. In the first test, we  compare our DG scheme with the finite element (FE) scheme related to the same time approximation. The DG scheme shows a well-behavior even for strong cross-diffusion effects  in contrast with FE where numerical spurious oscillations appear. Moreover, the second test exhibits the behavior of the tumor-growth model under different choices of parameters and also of mobility and proliferation functions.

\end{abstract}


\section{Introduction}

Lately, significant work on the mathematical modeling of tumor growth has been carried out. As a result, many different models have arisen, some of which have even been applied to predict the response of the tumor to its surrounding environment and possible medical treatments. Most of these models can be classified into micro-scale discrete models, macro-scale continuum models or hybrid models, \cite{lowengrub_et_al,cristini_lowengrub}. Regarding the continuum models, different approaches has been developed among which we can find models using both ODE, for instance, \cite{chulian2022mathematical,mcdougall}, and PDE, for example, \cite{fernandez2022glioblastoma,roose2007mathematical}.

In this sense, phase field models such as the Cahn-Hilliard (CH) equation have become a very popular tool. This model describes the evolution of a thin, diffuse, interface between two different phases or states of a process \cite{cahn_hilliard,novick_cohen_cahn_hilliard} through a so-called phase-filed variable, which minimizes an adequate free energy. Sometimes, this CH model is coupled with a degenerate mobility to impose phase-related pointwise bounds on this variable.

In particular, in the context of tumor modeling, the phase-field variable $u$ is usually interpreted as a tumor volume-fraction (with $0\le u \le 1$) and this model is coupled with other equations describing the interaction between the tumor and the surrounding environment. Some examples of these tumor models can be found in \cite{aristotelous,frieboes,wise,hawkins-daarud_numerical_2012,signoriunderstanding,garcke2022viscoelastic,fritz2023tumor} and the references therein. Often, certain physical properties, inherited from the Cahn-Hilliard equation, are inherent to the solution of these models such as mass-conservation of a biological substance, pointwise bounds on some of the variables 
and some sort of energy-dissipation.

In this work, we consider the model \eqref{problema:original_model}  carefully derived from mixture theory by Hawkins-Daarud et al. in \cite{hawkins-daarud_numerical_2012}, which describes the interaction between a tumor and the nutrients in the extracellular water. To this aim, the CH equation for $u$ is coupled with a diffusion equation for the nutrients $n$ by means of  some reaction and cross-diffusion terms. Although this model does not take into account some of the complex processes involved in the surrounding environment of the tumor, it allowed the authors to capture some irregular growth patterns that are typically associated with these processes. However, while this model is mass-conservative and energy-dissipative, it does not implicitly impose the necessary pointwise bounds on the volume fraction variables.

Therefore, we propose a modification of the aforementioned tumor model, see \eqref{problema:van_der_zee_orden_2} below, in accordance with its physical interpretation. As a result of this modification, we obtain pointwise bounds on the tumor and the nutrient volume fractions ($0\le u,n\le 1$) which are consistent with the physical meaning of the variables. This modification may help to a future application of this model (or a variant of it) for real tumor growth prediction.

This phase-field tumor model \eqref{problema:van_der_zee_orden_2} consist of a system of coupled nonlinear equations where reaction and cross diffusion effects appear. Thus, dealing with this model is really challenging both from a theoretical and the computational point of view.

In the case of the Cahn-Hilliard equation itself, several advances have been published regarding the existence and regularity of solution, most of which can be found in \cite{miranville2019cahn} and the references therein. Also, one can find several results regarding the existence, regularity and long-time behavior of the solution of the tumor model \eqref{problema:original_model} and variants in the literature in the case without cross-diffusion, see \cite{colli2014cahn,colli2015vanishing,colli2017asymptotic,frigeri2015diffuse}. Recently, the well-posedness and the long-time behavior of the model \eqref{problema:original_model} with cross-diffusion have been addressed in the work by H. Garcke and S. Yayla, \cite{garcke2020long}.

Regarding the numerical approximation of these equations, significant advances have been done both with respect to the time and the spatial discretizations.

On the one hand, the classical approach for the time discretization of the phase-field models is the convex-splitting decomposition introduced in \cite{eyre_1998_unconditionally} which preserves the energy stability. Nonetheless, other time-discrete schemes have been introduced in the literature (see, for instance, \cite{guillen-gonzalez_linear_2013,GUILLENGONZALEZ2014821,tierra2015numerical}). Among these time approximations we find the idea of introducing a Lagrange multiplier in the potential term in \cite{badia2011finite} which was extended in \cite{guillen-gonzalez_linear_2013,GUILLENGONZALEZ2014821,tierra2015numerical}. This idea led to the popular energy quadratization (EQ) schemes \cite{yang2016linear,yang2017efficient,yang2017numerical}, later extended to the scalar auxiliary variable (SAV) approach \cite{shen2018scalar}.

On the other hand, in the case of the Cahn-Hilliard equation with degenerate mobility, designing a suitable spatial discretization consistent with the physical properties of the model, specially the pointwise bounds, is a difficult task and only a few works have been published in this regard. Among the currently available structure-preserving schemes we can find some schemes based on finite volumes, \cite{bailo-carrillo_Cahn-Hilliard-FV_2021,huang2023structurepreserving}, and on finite elements, \cite{guillen-gonzalez_linear_2013}. Moreover, the authors have published a recent work, \cite{acosta-soba_CH_2022}, where the pointwise bounds of the CH model in the case with convection are preserved using an upwind discontinuous Galerkin (DG) approximation. To the best knowledge of the authors, no previous work has been published defining a fully discrete DG  scheme preserving the mass conservation, pointwise bounds and energy stability of the CH model with degenerate mobility.

The difficulties of the discretization are emphasized in the case of phase-field tumor models. In particular, in \cite{hawkins-daarud_numerical_2012} an energy-stable finite element scheme with a first-order convex-splitting scheme in time is proposed for \eqref{problema:original_model} and extended in \cite{wu_stabilized_2014} to a second-order time discretization. Other types of approximations of this model \eqref{problema:original_model} using meshless collocation methods, \cite{dehghan2017comparison}, stabilized element-free Galerkin method, \cite{mohammadi2019simulation}, and SAV Fourier-spectral method, \cite{shen2022sav}, can be found in the literature. However, no bounds are imposed on the discrete variables whatsoever.

In this sense, we introduce a well-suited convex-splitting DG scheme of the proposed model \eqref{problema:van_der_zee_orden_2}, based on the previous works \cite{acosta-soba_CH_2022,acosta-soba_KS_2022,mazen_saad_2014}. 
 This approximation preserves the physical properties of the phase-field tumor model (mass conservation, pointwise bounds and energy stability) and prevent numerical spurious oscillations. This scheme can be applied, in particular, to the more simple CH model with degenerate mobility itself preserving all of the aforementioned properties.

This paper is organized as follows: in Section \ref{sec:tumor_model} we discuss the tumor model \eqref{problema:original_model}, which was derived in \cite{hawkins-daarud_numerical_2012}, and we introduce our modified version of this model, \eqref{problema:van_der_zee_orden_2}, showing its physical properties. In Section \ref{sec:numerical_approx} we develop our numerical approximation of the tumor model \eqref{problema:van_der_zee_orden_2}. We introduce the convex-splitting time-discrete scheme \eqref{esquema:time_discrete} in Subsection \ref{sec:time_discrete}. Moreover, we present the DG space approximation, \eqref{esquema_DG_upw_Eyre_van_der_Zee}, in Subsection \ref{sec:fully_discrete_scheme}, defining the upwind form \eqref{def:aupw} in Subsection \ref{sec:def_aupw}. Then, we analyze the properties of the fully discrete scheme in Subsection \ref{sec:properties_DG_scheme}. Finally, we compute a couple of numerical experiments in Section \ref{sec:numer-experiments}. Specifically, in Subsection \eqref{sec:numer-experiments_1}, we present a numerical comparison between the robust DG scheme \eqref{esquema_DG_upw_Eyre_van_der_Zee} and a FE discretization of \eqref{esquema:time_discrete}, the latter of which fails in the case of strong cross-diffusion. In Subsection \ref{sec:numer-experiments_2}, we show the behavior of the model \eqref{problema:van_der_zee_orden_2} under different choices of parameters and mobility/proliferation functions.

\section{Modified tumor model}
\label{sec:tumor_model}

Let $\Omega\subset \R^d$ be a bounded polygonal domain and $T>0$ the final time.

The following tumor-growth model was introduced in \cite{hawkins-daarud_numerical_2012} and further studied in \cite{wu_stabilized_2014}:
\begin{subequations}
	\label{problema:original_model}
	\begin{align}
		\label{eq:original_model_1}
		\partial_t u&=\nabla\cdot\left(M_u\nabla\mu_u\right)+\delta P(u)(\mu_n-\mu_u) \quad&\text{in }\Omega\times (0,T),\\
		\label{eq:original_model_2}
		\mu_u&=F'(u)-\varepsilon^2\Delta u-\chi_0 n\quad&\text{in }\Omega\times (0,T),\\
		\label{eq:original_model_3}
		\partial_t n&=\nabla\cdot\left(M_n\nabla\mu_n\right)-\delta P(u) (\mu_n-\mu_u) \quad&\text{in }\Omega\times (0,T),\\
		\nabla u\cdot \mathbf{n}&=\left( M_n\nabla \mu_n\right)\cdot \mathbf{n}=\left( M_u\nabla \mu_u\right)\cdot \mathbf{n}=0 \quad &\text{on }\partial\Omega\times (0,T),\\
		u(0)&=u_0,\quad n(0)=n_0\quad&\text{in }\Omega,
	\end{align}
\end{subequations}
where
\begin{equation}
	\label{eq:original_model_4}
	\mu_n=\frac{1}{\delta} n -\chi_0 u \quad\text{in }\Omega\times (0,T)
\end{equation}
 and $u_0,n_0\in L^2(\Omega)$. In this model, $u$ and $n$ represent the tumor cells and the nutrient-rich extracellular water volume fractions, respectively. Therefore, these variables are assumed to be bounded in $[0,1]$. Moreover, $\mu_u$ and $\mu_n$ are the (chemical) potentials of $u$ and $n$, respectively.

The behavior of the cells is modeled using a Cahn-Hilliard equation, where 
$$F(u)=\frac{1}{4}u^2(1-u)^2$$ is the Ginzburg-Landau double well potential and $M_u$ is the mobility of the tumor, which is taken either as constant or degenerated at $u=0$, for instance, $M_u(u)=\widehat M u^2$ with $\widehat M>0$. The parameter $\varepsilon\ge 0$ is related to the thickness of the interface between the tumor phases $u=1$ (fully saturated) and $u=0$ (fully unsaturated).

On the other hand, the nutrients are modeled using a diffusion equation where the function $M_n$ is the mobility of the nutrients, which is taken as constant in practice.

These equations are coupled by cross diffusion terms (multiplied by the coefficient $\chi_0\ge 0$) introduced in \eqref{eq:original_model_2} and \eqref{eq:original_model_4} that model the attraction between tumor cells and nutrients. In addition, reaction terms modeling the consumption of nutrients by the tumor cells appear in \eqref{eq:original_model_1} and \eqref{eq:original_model_3}, where $P(u)$ is a proliferation term that vanishes when $u\le0$ in \cite{hawkins-daarud_numerical_2012} or when $u\not\in(0,1)$ in \cite{wu_stabilized_2014}. These reaction terms depend on the difference between the potentials, which is assumed to be positive as the parameter $\delta>0$ is very small, because one has the approximation
$$
\delta P(u)(\mu_n-\mu_u) =P(u)(n-\delta(\chi_0 u-\mu_u)) \approx P(u)n
\quad  \hbox{if $\delta \approx  0$.}
$$

The well-posedness and long-time behavior of the model \eqref{problema:original_model} and some variants have been considered in the case without cross-diffusion ($\chi_0=0$), see \cite{frigeri2015diffuse,colli2014cahn,colli2015vanishing,colli2017asymptotic}, and only recently in the case with cross diffusion ($\chi_0>0$) in \cite{garcke2020long}.

In this work, taking into account the previous considerations, we introduce the following modified phase field tumor model
\begin{subequations}
	\label{problema:van_der_zee_orden_2}
	\begin{align}
		\label{eq:van_der_zee_orden_2_1}
		\partial_t u&=C_u\nabla\cdot\left(M(u)\nabla\mu_u\right)+\delta P_0P(u,n)(\mu_n-\mu_u)_\oplus \quad&\text{in }\Omega\times (0,T),\\
		\label{eq:van_der_zee_orden_2_2}
		\mu_u&=F'(u)-\varepsilon^2\Delta u-\chi_0 n\quad&\text{in }\Omega\times (0,T),\\
		\label{eq:van_der_zee_orden_2_3}
		\partial_t n&=C_n\nabla\cdot\left(M(n)\nabla\mu_n\right)-\delta P_0P(u,n)(\mu_n-\mu_u)_\oplus \quad&\text{in }\Omega\times (0,T),\\
		\nabla u\cdot \mathbf{n}&=\left( M(n)\nabla \mu_n\right)\cdot \mathbf{n}=\left( M(u)\nabla \mu_u\right)\cdot \mathbf{n}=0 \quad &\text{on }\partial\Omega\times (0,T),\\
		u(0)&=u_0,\quad n(0)=n_0\quad&\text{in }\Omega,
	\end{align}
\end{subequations}
where
\begin{equation}
	\label{eq:van_der_zee_orden_2_4}
	\mu_n=\frac{1}{\delta} n -\chi_0 u \quad\text{in }\Omega\times (0,T),
\end{equation}
$u_0,n_0\in L^2(\Omega)$, $F(u)=\frac{1}{4}u^2(1-u)^2$ and all the parameters above are nonnegative with $\delta, C_u, C_n>0$ and $\varepsilon, \chi_0, P_0\ge 0$. Also, we have introduced the following notation for the positive and negative parts of a scalar function $v$: 
$$
v_\oplus \coloneqq\frac{|v|+v}{2}=\max\{v,0\},
\quad
v_\ominus \coloneqq\frac{|v|-v}{2}=-\min\{v,0\},
\quad
v=v_\oplus  - v_\ominus \, .
$$
Moreover, we define the following family of degenerate mobilities
\begin{equation}
	\label{def:mobility}
	M(v)\defeq h_{p,q}(v),
\end{equation}
for certain $p,q\in\N$ where
$$
h_{p,q}(v)\defeq K_{p,q}v_\oplus^p(1-v)_\oplus^q=
\begin{cases}
	K_{p,q}v^p(1-v)^q, & v\in[0,1],\\
	0, &\text{elsewhere},
\end{cases}
$$
with $K_{p,q}>0$ a constant so that $\max_{x\in\R}h_{p,q}(v)=1$, hence $M(v)$ is a degenerate and normalized mobility. In addition, we define the proliferation function depending on both cells and nutrients as
\begin{equation}
	\label{def:proliferation}
	P(u,n)\defeq h_{r,s}(u)n_\oplus,
\end{equation}
for certain $r,s\in\N$.

Notice that the mobility functions, defined in \eqref{def:mobility}, for the tumor and for the nutrients do not necessarily need to be identical. One may consider the tumor mobility as $M_u(u)=h_{p,q}(u)$ with $p, q\in\N$ and the nutrients mobility as $M_n(n)=h_{p',q'}(n)$ with $p', q'\in\N$ and all the results below equally hold. However, for simplicity, we will assume that $M_u=M_n$ and denote the mobility function as $M$.

\begin{remark}
	\label{rmk:modifications_model}
	This model introduces several changes with respect to the previous model \eqref{problema:original_model} studied in \cite{hawkins-daarud_numerical_2012,wu_stabilized_2014}. 
  	These modifications, described next, involve significant improvements. But also they lead to a more complex model, with degenerate mobility and proliferation functions and major difficulties regarding the analysis of the existence, regularity and long time behavior of solutions.

  	Specifically:
	\begin{itemize}
		\item The difference between the potentials, $\mu_n-\mu_u$ is assumed to be positive since $\delta$ is set to be a very small parameter. This difference could possible be negative in the regions where $n\simeq 0$ but, in this case, the reaction terms vanish due to the proliferation function $P(u,n)$ defined in \eqref{def:proliferation}. Therefore, the positive part of $(\mu_n-\mu_u)$ is taken in \eqref{eq:van_der_zee_orden_2_1} and \eqref{eq:van_der_zee_orden_2_3}.
		\item  When  $\delta\to 0$, the reaction terms in equations \eqref{eq:van_der_zee_orden_2_1} and \eqref{eq:van_der_zee_orden_2_3} are assumed to grow with the square of the nutrients volume fraction. In fact
			$$
		\delta P_0 P(u,n)(\mu_n-\mu_u)_\oplus =P_0 P(u,n)(n-\delta(\chi_0 u-\mu_u))_\oplus \simeq P_0P(u,n)n= P_0 h_{r,s}(u)(n_\oplus)^2.
		$$
		\item A degenerate mobility, \eqref{def:mobility}, is considered for both the phase-field function $u$ and the volume fraction of nutrients $n$.
    \item The aforementioned modifications imply that $u$ and $n$ must be bounded in the interval $[0,1]$ (see Theorem \ref{prop:maximum_principle}), what matches the physical assumptions of the model since $u$ and $n$ are assumed to be volume fractions. This is a clear improvement over previous approaches, such us the ones considered in \cite{hawkins-daarud_numerical_2012,wu_stabilized_2014}, where the solution does not necessarily satisfy these bounds.
	\end{itemize}
\end{remark}

\begin{remark}
	In practice, $C_n=\delta D$ with $D>0$ so that, when $\delta\to0$, the $n$-equation is approached by 
	$$
	\partial_t n\approx D\nabla\cdot\left(M(n)\nabla n\right)-P_0P(u,n)n.
	$$
\end{remark}

Considering that $\mu_n$ is explicitly determined by \eqref{eq:van_der_zee_orden_2_4}, we can reduce the number of unknowns and define the weak formulation of \eqref{problema:van_der_zee_orden_2} as: find $(u,\mu_u,n)$ such that $u, n\in L^2(0,T;H^1(\Omega))$, $\partial_t u,\partial_t n\in L^2(0,T;H^1(\Omega)')$ and $\mu_u\in L^2(0,T; H^1(\Omega))$,
which satisfies the following variational problem a.e. $t\in(0,T)$
\begin{subequations}
	\label{problema:van_der_zee_form_var_def}
	\begin{align}
		\label{eq:van_der_zee_form_var_def_1}
		\dualH{\partial_t u(t)}{\overline{u}}&=-C_u\escalarLd{M(u(t))\nabla\mu_u(t)}{\nabla\overline{u}}\notag\\&\quad+\delta P_0 \escalarL{P(u(t),n(t))(\mu_n(t)-\mu_u(t))_\oplus }{\overline{u}},&&\forall\overline{u}\in H^1(\Omega),\\
		\label{eq:van_der_zee_form_var_def_2}
		\escalarL{\mu_u(t)}{\overline{\mu}_u}&=\varepsilon^2 \escalarLd{\nabla u(t)}{\nabla\overline{\mu}_u}+\escalarL{F'(u(t))-\chi_0 n(t)}{\overline{\mu}_u},&&\forall\overline{\mu}_u\in H^1(\Omega),\\
		\label{eq:van_der_zee_form_var_def_3}
		\dualH{\partial_t n(t)}{\overline{n}}&=-C_n\escalarLd{M(n(t))\nabla\mu_n(t)}{\nabla\overline{n}}\notag\\&\quad-\delta P_0 \escalarL{P(u(t),n(t))(\mu_n(t)-\mu_u(t))_\oplus }{\overline{n}},&&\forall\overline{n}\in H^1(\Omega),
	\end{align}
\end{subequations}
where 
\begin{equation}
	\label{eq:van_der_zee_form_var_def_4}
	\mu_n(t)=\frac{1}{\delta} n(t) -\chi_0 u(t),
\end{equation}
$u(0)=u_0$, $n(0)=n_0$ and $\escalarL{\cdot}{\cdot}$, $\dualH{\cdot}{\cdot}$ denote the usual scalar product in $L^2(\Omega)$ and the dual product over $H^1(\Omega)$, respectively.

Since $u, n\in L^2(0,T, H^1(\Omega))$ with $\partial_t u, \partial_t n\in L^2(0,T, H^1(\Omega)')$, it is known, see for instance \cite{Ern-Guermond:04,dautray1999mathematical}, that $u, n\in \mathcal{C}^0([0,T], L^2(\Omega))$ and that $\dualH{\partial_t u(t)}{\overline{u}} = \frac{d}{dt}\escalarL{u(t)}{\overline{u}}$, $\dualH{\partial_t n(t)}{\overline{n}} = \frac{d}{dt}\escalarL{n(t)}{\overline{n}}$ for a.e. $t\in(0,T)$ and every $\bu,\bn\in H^1(\Omega)$.

\begin{proposition}
	\label{prop:maximum_principle}
	Given $u_0,v_0\in[0,1]$, any solution $(u,\mu_u,n)$ of the model \eqref{problema:van_der_zee_form_var_def} satisfies that $u(t)$ and $n(t)$ are bounded in $[0,1]$ for a.e. $t\in(0,T)$.
\end{proposition}
\begin{proof}
	Let $(u,\mu_u,n)$ be a solution of the model \eqref{problema:van_der_zee_form_var_def} and $u_0,v_0\in[0,1]$.
	\begin{itemize}
		\item First, we prove that $u, n\ge 0$. Notice that $u_\ominus\in L^2(0,T, H^1(\Omega))$ and take $\overline{u}=u(t)_\ominus$ for a.e. $t\in(0,T)$ in \eqref{eq:van_der_zee_form_var_def_1}. We arrive at $\frac{1}{2}\frac{d}{dt}\normaL{u(t)_\ominus}^2= 0$, hence $\normaL{u(t)_\ominus}=\normaL{u(0)_\ominus}=0$. Similarly, $\normaL{n(t)_\ominus}=0$ for a.e. $t\in(0,T)$.
		\item Now, we prove that $u, n\ge 1$. Notice that $(1-u)_\ominus\in L^2(0,T, H^1(\Omega))$, $\partial_t u = \partial_t (u-1)$ and take $\overline{u}=(u(t)-1)_\oplus$ for a.e. $t\in(0,T)$ in \eqref{eq:van_der_zee_form_var_def_1}. We arrive at $\frac{1}{2}\frac{d}{dt}\normaL{(u(t)-1)_\oplus}^2= 0$, hence $\normaL{(u(t)-1)_\oplus}=\normaL{(u(0)-1)_\oplus}=0$. Similarly, $\frac{1}{2}\frac{d}{dt}\normaL{(n(t)-1)_\oplus}^2\le 0$ what implies $\normaL{(n(t)-1)_\oplus}\le \normaL{(n(0)-1)_\oplus}=0$ for a.e. $t\in(0,T)$.
	\end{itemize}
\end{proof}

\begin{proposition}
	\label{prop:mass_conservation_continous}
	Let $(u,\mu_u,n)$ be a solution of the problem \eqref{problema:van_der_zee_form_var_def}. Then, this solution conserves the total mass of tumor cells and nutrients in the sense of $$\frac{d}{dt}\int_\Omega (u(x,t)+n(x,t))dx=0.$$
\end{proposition}
\begin{proof}
	It is enough to take $\overline{u}=\overline{n}=1$ in \eqref{eq:van_der_zee_form_var_def_1} and \eqref{eq:van_der_zee_form_var_def_3} and add the resulting expressions.
\end{proof}
	
\begin{proposition}
	\label{prop:energy_law_continuous}
	If $(u,\mu_u,n)$ is a solution of the problem \eqref{problema:van_der_zee_form_var_def} with $\partial_t u\in L^2(0,T, H^1(\Omega))$, then it satisfies the following energy law
	\begin{align}
		\label{ley_energia_continua_van_der_zee}
		\frac{d E(u(t),n(t))}{dt}&
		+C_u\int_\Omega M(u(x,t))|\nabla\mu_u(x,t)|^2dx
		+C_n\int_\Omega M(n(x,t))|\nabla\mu_n(x,t)|^2 dx \notag
		\\&
		+\delta P_0\int_\Omega P(u(x,t),n(x,t))(\mu_u(x,t)-\mu_n(x,t))_\oplus ^2dx=0,
	\end{align}
	where the energy functional is defined by 
	\begin{align}
		\label{energia_van_der_zee}
		E(u,n)&\coloneqq\int_\Omega\left(\frac{\varepsilon^2}{2}|\nabla u|^2+ F(u)-\chi_0 u\, n 
		+\frac{1}{2\delta}n^2\right).
	\end{align}
	Therefore, the solution is energy stable in the sense $$\frac{d}{dt}E(u(t),n(t))\le 0.$$
\end{proposition}
\begin{proof}
	Take $\bu=\mu_u(t)$, $\bmu_u=\partial_t u(t)$, $\bn=\mu_n(t)$ in \eqref{eq:van_der_zee_form_var_def_1}--\eqref{eq:van_der_zee_form_var_def_3} and test \eqref{eq:van_der_zee_form_var_def_4} with $\partial_t n(t)$.
	Adding the resulting expressions we arrive at
	\begin{multline*}
		\varepsilon^2\escalarLd{\nabla u(t)}{\nabla(\partial_t u(t))} + \escalarL{F'(u(t))}{\partial_t u(t)} -\chi_0\left[\escalarL{n(t)}{\partial_tu(t)} + \escalarL{u(t)}{\partial_t n(t)}\right]+\frac{1}{\delta}\escalarL{n(t)}{\partial_t n(t)}\\+C_u\int_\Omega M(u(x,t))\vert\nabla\mu_u(t)\vert^2dx+C_n\int_\Omega M(n(x,t))\vert\nabla\mu_n(t)\vert^2dx\\+\delta P_0\int_\Omega P(u(x,t),n(x,t))(\mu_u(x,t)-\mu_n(x,t))_\oplus (\mu_u(x,t)-\mu_n(x,t))dx=0.
	\end{multline*}
	Therefore, it is straightforward to check that \eqref{ley_energia_continua_van_der_zee} holds.
\end{proof}

\section{Numerical approximation}
\label{sec:numerical_approx}

In this section we will develop a well suited approximation of the tumor model \eqref{problema:van_der_zee_orden_2} which preserves the physical properties presented in the previous section.

\subsection{Time-discrete scheme}
\label{sec:time_discrete}

Regarding the time discretization, we take an equispaced partition $0=t_0<t_1<\cdots<t_N=T$ of the time domain $[0,T]$ with $\Delta t=t_{m+1}-t_m$ the time step. Also, given a scalar function $v$ defined on $[0,T]$ we will denote $v^m\simeq v(t_m)$ and $\delta_t v^{m+1}=(v^{m+1}-v^m)/\Delta t$ the time-discrete derivative.

Now, we define a convex splitting of the double well potential $F(u)$ as follows:
$$
F(u)\coloneqq F_i(u)+F_e(u),\quad F_i(u)\coloneqq\frac{3}{8}u^2,\quad F_e(u)\coloneqq\frac{1}{4}u^4-\frac{1}{2}u^3-\frac{1}{8}u^2, \quad u\in[0,1],
$$
where we are going to treat the convex term, $F_i(u)$, implicitly and the concave term, $F_e(u)$, explicitly (see, for instance, \cite{eyre_1998_unconditionally, guillen-gonzalez_linear_2013,acosta-soba_CH_2022}, for more details). For this, we define
\begin{equation} \label{convex-split}
f(u^{m+1},u^m)\coloneqq F_i'(u^{m+1})+F_e'(u^m)
=\frac{1}{4}\left(3 u^{m+1}+4 (u^m)^3-6(u^m)^2-u^m\right). 
\end{equation}

We propose the following time-discrete scheme: given $(u^m,\mu_u^m,n^m)\in H^1(\Omega)^3$ 
with $u^m,n^m\in[0,1]$ in $\Omega$,  find $(u^{m+1},\mu_u^{m+1},n^{m+1}) \in H^1(\Omega)^3$ such that
\begin{subequations}
	\label{esquema:time_discrete}
	\begin{align}
		\label{eq:time_discrete_1}
		\escalarL{\delta_t u^{m+1}}{\overline{u}}&=-C_u\escalarLd{M(u^{m+1})\nabla\mu_u^{m+1}}{\nabla\overline{u}}&\notag\\&\quad+\delta P_0 \escalarL{P(u^{m+1},n^{m+1})(\mu_n^{m+1}-\mu_u^{m+1})_\oplus }{\overline{u}},\quad&\forall\overline{u}\in H^1(\Omega),\\
		\label{eq:time_discrete_2}
		\escalarL{\mu_u^{m+1}}{\overline{\mu}_u}&=\varepsilon^2 \escalarLd{\nabla u^{m+1}}{\nabla\overline{\mu}_u}&\notag\\&\quad+\escalarL{f(u^{m+1},u^m)-\chi_0 n^{m+1}}{\overline{\mu}_u},\quad&\forall\overline{\mu}_u\in H^1(\Omega),\\
		\label{eq:time_discrete_3}
		\escalarL{\delta_t n^{m+1}}{\overline{n}}&=-C_n\escalarLd{M(n^{m+1})\nabla\mu_n^{m+1}}{\nabla\overline{n}}&\notag\\&\quad-\delta P_0 \escalarL{P(u^{m+1},n^{m+1})(\mu_n^{m+1}-\mu_u^{m+1})_\oplus }{\overline{n}},\quad&\forall\overline{n}\in H^1(\Omega),
	\end{align}
\end{subequations}
where
\begin{equation}
	\label{eq:time_discrete_4}
	\mu_n^{m+1}=\frac{1}{\delta} n^{m+1} -\chi_0 u^m
\end{equation}
and $u^0=u_0$, $n^0=n_0$ in $\Omega$.

Notice that the proposed scheme \eqref{esquema:time_discrete} is just a variation of backward Euler's method where we have treated explicitly the concave part of the splitting of $F(u)$ in \eqref{eq:time_discrete_2} and a part of the cross diffusion in \eqref{eq:time_discrete_4}.

The following results are the discrete in time versions of the Propositions \ref{prop:maximum_principle}, \ref{prop:mass_conservation_continous} and \ref{prop:energy_law_continuous}. We skip the proofs, because the same properties will be proved for the fully discrete scheme given in Subsection \ref{sec:fully_discrete_scheme}.

\begin{proposition}
	\label{prop:maximum_principle_time-discrete}
	Any solution $(u^{m+1},\mu_u^{m+1},n^{m+1})$ of the time-discrete scheme \eqref{esquema:time_discrete} satisfies that $u^{m+1}, n^{m+1}\in [0,1]$ in $\Omega$.
\end{proposition}

\begin{proposition}
	\label{prop:mass_conservation_time-discrete}
	Any solution $(u^{m+1},\mu_u^{m+1},n^{m+1})$ of the time-discrete scheme \eqref{esquema:time_discrete} 
	 conserves the total mass of tumor cells and nutrients in the sense of $$\delta_t\int_\Omega (u^{m+1}(x)+n^{m+1}(x))dx=0.$$
\end{proposition}

\begin{proposition}
	\label{prop:ley_energia_time-discrete}
	Any solution $(u^{m+1},\mu_u^{m+1},n^{m+1})$ of the time-discrete scheme \eqref{esquema:time_discrete} satisfies the following discrete energy law
	\begin{align}
		\label{ley_energia_time-discrete}
		\delta_t E(u^{m+1},n^{m+1})&+C_u\int_\Omega M(u^{m+1})\vert\nabla\mu_u^{m+1}\vert^2
		+C_n\int_\Omega M(n^{m+1})|\nabla\mu_n^{m+1}|^2\notag
		\\&
		+\frac{\Delta t \, \varepsilon^2}{2}\int_\Omega|\delta_t\nabla u^{m+1}|^2
		+\frac{\Delta t}{2\delta}\int_\Omega \vert\delta_t n^{m+1}\vert^2\nonumber
		\\&
		+\delta \, P_0\int_\Omega P(u^{m+1},n^{m+1})(\mu_u^{m+1}-\mu_n^{m+1})_\oplus ^2\le 0,
	\end{align}
	where $E(u,n)$ is defined in \eqref{energia_van_der_zee}.
	
	Therefore, the solution is energy stable in the sense $$E(u^{m+1},n^{m+1})\le E(u^{m},n^{m}),\quad \forall\, m\ge 0.$$
\end{proposition}

\begin{remark}
	This first order nonlinear time-discrete scheme \eqref{esquema:time_discrete} preserves the properties of the continuous models, namely  the
	conservation $\int_\Omega (u^{m+1} + n^{m+1}) = \int_\Omega (u^{m} + n^{m})$, the point-wise bounds  $u^{m+1},n^{m+1}\in[0,1]$ in $\Omega$ and the energy dissipation $E(u^{m+1},n^{m+1})\le E(u^{m},n^{m})$. In particular, the
	 convex-splitting technique used in 
	 \eqref{convex-split} 
	  allows us to guarantee the dissipation of the discrete version of the exact energy of the model \eqref{energia_van_der_zee}.

		There exist other widely used techniques such as the SAV approach which does rely on the dissipation of a modified energy via an auxiliary variable that must be introduced in the scheme. An advantage of the SAV approach is that it usually leads, in more simple models, to linear schemes with decreasing modified energy but, in this case, since we want to preserve the pointwise bounds, we require anyway a nonlinear discretization as shown in \eqref{esquema:time_discrete}.

	Therefore, in this context, we prefer the convex-splitting technique against the SAV approach. However, it would be interesting to explore whether it is possible to extend the ideas of this work to design a second-order in time approximation, which is not straightforward and may require a different approach such as a SAV-type discretization.
\end{remark}

\subsection{Fully discrete scheme}
\label{sec:fully_discrete_scheme}

For the space discretization, we consider a shape-regular  triangular mesh $\T_h=\{K\}_{K\in \T_h}$ of size $h$ over $\Omega$, and we note by $\E_h$ the set of the edges of $\T_h$ where $\E_h^{\text{i}}$ denotes the \textit{interior edges} and $\E_h^{\text{b}}$ denotes the \textit{boundary edges} such that $\E_h=\E_h^{\text{i}}\cup\E_h^{\text{b}}$.

We set the following mesh orientation: the unit normal vector $\nn_e$ associated to an interior edge $e\in\E_h^{\text{i}}$ shared by the elements $K, L\in\T_h$, i.e. $e=\partial K\cap\partial L$, is exterior to $K$ pointing to $L$. Moreover, for the boundary edges $e\in\E_h^{\text{b}}$, the unit normal vector $\nn_e$ points outwards of the domain $\Omega$.

In addition, we assume the following hypothesis:
\begin{hyp}
	\label{hyp:mesh}
	The line between the barycenters of any adjacent triangles $K$ and $L$ is orthogonal to the interface $e=K\cap L\in\E_h^{\text{i}}$.
\end{hyp}
One example of a mesh satisfying Hypothesis~\ref{hyp:mesh} is plotted in Figure~\ref{fig:mesh}. For other examples and a further insight on this property we refer the reader to \cite{acosta-soba_KS_2022}.

We define the \textit{average} $\media{\cdot}$ and the \textit{jump} $\salto{\cdot}$ of a scalar function $v$ on an edge $e\in\E_h$ as follows:
\begin{equation*}
		\media{v}\coloneqq
		\begin{cases}
			\dfrac{\vK+\vL}{2}&\text{if } e\in\E_h^{\text{i}}\\
			\vK&\text{if }e\in\E_h^{\text{b}}
		\end{cases},
		\qquad
		\salto{v}\coloneqq
		\begin{cases}
			\vK-\vL&\text{if } e\in\E_h^{\text{i}}\\
			\vK&\text{if }e\in\E_h^{\text{b}}
		\end{cases}.
\end{equation*}

Let $\Pd_k(\T_h)$ and $\Pc_k(\T_h)$ be the spaces of Finite Element discontinuous and continuous functions, respectively, whose restriction to the elements $K$ of $\T_h$ are polynomials of degree $k\ge 0$. Also, we define the projection $\Pi_0\colon L^1(\Omega) \rightarrow \Pd_0(\T_h)$ and the regularization $\Pi^h_1\colon L^1(\Omega)\rightarrow \Pc_1(\T_h)$ of a function $g\in L^1(\Omega)$ as the function satisfying the following:
\begin{align}
	\label{eq:esquema_DG_Pi0}
	\escalarL{g}{\overline{w}}&=
	\escalarL{\Pi_0 g}{\overline{w}},&\forall\,\overline{w}\in \Pd_0(\T_h),
	\\
	\label{eq:esquema_DG_Pih1}
\escalarL{g}{\overline{\phi}}&=\escalarML{\Pi^h_1 g}{\overline{\phi}},&\forall\,\overline{\phi}\in \Pc_1(\T_h),
\end{align}
where $\escalarML{\cdot}{\cdot}$ is the mass-lumping scalar product in $\Pc_1(\T_h)$. 
In fact, $(\Pi_0 g)|_K =( \int_K g)/|K|$  for all $K\in \T_h $, 
and 
$(\Pi^h_1 g)(a_j)=(\sum_{K\in \text{Sop} (a_j)} \int_K g\ \varphi_j) /(\sum_{K\in \text{Sop} (a_j)}|K|/(d+1) )$
 for all vertex $a_j$ with $\varphi_j $ the canonical basis of $\Pc_1(\T_h)$.
For a further insight on discontinuous Galerkin methods we recommend \cite{di_pietro_ern_2012}.

\

We propose the following fully discrete scheme for the model \eqref{problema:van_der_zee_orden_2}: given $u^m,n^m\in\Pd_0(\T_h)$ with $u^{m},n^{m}\in[0,1]$ and $\mu_u^m\in\Pc_1(\T_h)$, find $u^{m+1}, n^{m+1}\in \Pd_0(\T_h)$ and $\mu_u^{m+1} \in \Pc_1(\T_h)$, such that
\begin{subequations}
	\label{esquema_DG_upw_Eyre_van_der_Zee}
	\begin{align}
		\label{eq:esquema_DG_upw_Eyre_van_der_Zee_1}
		\escalarL{\delta_tu^{m+1}}{\overline{u}}&=-C_u\aupw{\Pi_0\mu_u^{m+1}}{M(u^{m+1})}{\overline{u}}&\notag\\&\quad+\delta P_0\escalarL{P(u^{m+1},n^{m+1})(\mu_n^{m+1}-\Pi_0\mu_u^{m+1})_\oplus }{\overline{u}},\quad&\forall\overline{u}\in\Pd_0(\T_h),\\
		\label{eq:esquema_DG_upw_Eyre_van_der_Zee_2}
		\escalarML{\mu_u^{m+1}}{\overline{\mu}_u}&=\varepsilon^2 \escalarLd{\nabla \up^{m+1}}{\nabla\overline{\mu}_u}+ \escalarL{f(\up^{m+1},\up^m)}{\overline{\mu}_u}\nonumber\\&\quad-\chi_0\escalarL{ n^{m+1}}{\overline{\mu}_u},\quad&\forall\overline{\mu}_u\in \Pc_1(\T_h),\\
		\label{eq:esquema_DG_upw_Eyre_van_der_Zee_3}
		\escalarL{\delta_t n^{m+1}}{\overline{n}}&=-C_n\aupw{\mu_n^{m+1}}{M(n^{m+1})}{\overline{n}}&\notag\\&\quad-\delta P_0 \escalarL{P(u^{m+1},n^{m+1})(\mu_n^{m+1}-\Pi_0\mu_u^{m+1})_\oplus }{\overline{n}},\quad&\forall\overline{n}\in \Pd_0(\T_h),
	\end{align}
\end{subequations}
where 
\begin{equation}
	\label{eq:esquema_DG_upw_Eyre_van_der_Zee_4}
	\mu_n^{m+1}=\frac{1}{\delta}n^{m+1}  -\chi_0 \Pi_0(\up^{m}),
\end{equation}
 $u^0=u_0$, $n^0=n_0$ and $\aupw{\cdot}{\cdot}{\cdot}$ is an upwind form defined in Subsection \ref{sec:def_aupw} below. 
Note that, 
\begin{equation} \label{Pi_1}
(\Pi^h_1 u^m)(a_j)=\frac{\sum_{L\in \text{Sop} (a_j)} |L| u^m_L} {\sum_{L\in \text{Sop} (a_j)}|L| },
\quad
\forall\, a_j,
\end{equation}
and
\begin{equation} \label{Pi_0Pi_1}
\Pi_0(\up^{m})|_K=\frac1{d+1} \sum_{a_j\in K}(\Pi^h_1 u^m)(a_j) ,
\quad
 \forall\, K\in \T_h.
\end{equation}

To ease the notation, we denote the solution of this fully discrete scheme in the same way than the time discrete scheme \eqref{esquema:time_discrete}. From now on we will refer to the solution of the fully discrete scheme unless otherwise specified.   

Notice that we have introduced the regularization of $u^{m+1}$, $\up^{m+1}$ to preserve the diffusion term in \eqref{eq:esquema_DG_upw_Eyre_van_der_Zee_2}.
In fact, this regularized variable will be regarded as our approximation of the tumor cells volume fraction as, according to the results in Subsection \ref{sec:properties_DG_scheme}, it preserves the maximum principle and satisfies a discrete energy law. Moreover, in order to preserve the maximum principle and the dissipation of the energy, we consider mass lumping in the term $\escalarML{\mu_u^{m+1}}{\overline{\mu}}$.

\begin{remark}
	The homogeneous Neumann boundary conditions on $u^m$ and $n^m$ have been implicitly imposed in the definition of $\aupw{\cdot}{\cdot}{\cdot}$, see \eqref{def:aupw}. In addition, the boundary condition $\nabla \up^{m}\cdot\nn=0$ on $\partial\Omega\times(0,T)$ is imposed implicitly by the term $\escalarL{\nabla \up^{m}}{\nabla\overline{\mu}}$ in \eqref{eq:esquema_DG_upw_Eyre_van_der_Zee_2}.
\end{remark}

\begin{remark}
	The scheme \eqref{esquema_DG_upw_Eyre_van_der_Zee} is nonlinear so we will have to use an iterative procedure, such as Newton's method, to approach its solution.
\end{remark}

\subsubsection{Definition of $\aupw{\cdot}{\cdot}{\cdot}$}
\label{sec:def_aupw}

First of all, following the ideas in \cite{acosta-soba_CH_2022}, in order to preserve the maximum principle using an upwind approximation we will split the mobility function into its increasing and its decreasing part as follows:
$$
M^\uparrow(v)=
\begin{cases}
	M(v),& v\le v^*,\\
	M(v^*),& v>v^*,
\end{cases}
\quad
M^\downarrow(v)=
\begin{cases}
	0,& v\le v^*,\\
	M(v)-M(v^*),& v>v^*,
\end{cases}
$$
where  $v^*\in\R$ is the point where the maximum of $M(v)$ is attained, which can be obtained
by simple algebraic computations. Note that 
 $M(v)=M^\uparrow(v)+M^\downarrow(v)$.

Now, we define the following upwind form for $v,\bv,\mu\in\P_0(\T_h)$:
\begin{multline}
	\label{def:aupw}
	\aupw{\mu}{M(v)}{\bv}\defeq\\ \sum_{e\in\E_h^i,e=K\cap L}\int_e\left(\left(-\nabland\mu\right)_\oplus \left(M^\uparrow(\vK) + M^\downarrow(\vL)\right)_\oplus-(-\nabland\mu)_\ominus \left(M^\uparrow(\vL) + M^\downarrow(\vK)\right)_\oplus\right)\salto{\bv}
\end{multline}
with
\begin{equation}
\label{eq:approx_gradn}
\nabland\mu =\frac{-\salto{\mu}}{\mathcal{D}_e(\T_h)}= \frac{\muL-\muK}{\mathcal{D}_e(\T_h)},
\end{equation}
a reconstruction of the normal gradient using $\P_0(\T_h)$ functions for every $e\in\E_h^i$ with $e=K\cap L$ (see \cite{acosta-soba_KS_2022} for more details). We have denoted $\mathcal{D}_e(\T_h)$ the distance between the barycenters of the triangles $K$ and $L$ of the mesh $\T_h$ that share $e\in\Ehi$.
This way, we can rewrite \eqref{def:aupw} as
\begin{multline}
\label{def:aupw_barycenter}
\aupw{\mu}{M(v)}{\bv}\defeq \\\sum_{e\in\E_h^i,e=K\cap L}\frac{1}{\mathcal{D}_e(\T_h)}\int_e\left(\salto{\mu}_\oplus \left(M^\uparrow(\vK) + M^\downarrow(\vL)\right)_\oplus-\salto{\mu}_\ominus \left(M^\uparrow(\vL) + M^\downarrow(\vK)\right)_\oplus\right)\salto{\bv}.
\end{multline}

\begin{remark}
The form $\aupw{\mu}{M(v)}{\bv}$ is an upwind approximation of the convective term
$$
-\escalarL{M(v)\nabla\mu}{\nabla\bv},\quad \bv\in H^1(\Omega),
$$
taking into consideration that the orientation of the flux is determined by both the orientation of $\nabla\mu$ and the sign of $M'(v)$ as follows:
$$
\nabla\cdot(M(v)\nabla\mu)=M'(v)\nabla v\nabla\mu+M(v)\Delta \mu.
$$
In order to develop this approximation we have followed the ideas of \cite{acosta-soba_CH_2022,acosta-soba_KS_2022,mazen_saad_2014} and we have considered an approximation of $M(v)$ in \eqref{def:aupw} by means of its increasing and decreasing parts, $M^\uparrow(v)$ and $M^\downarrow(v)$, whose derivatives are positive and negative, respectively. However, unlike in \cite{acosta-soba_CH_2022}, we have also truncated the mobility $M(v)$
to avoid negative approximations of $M(v)$ that may lead to a loss of energy stability.
\end{remark}

\subsubsection{Properties of the fully discrete scheme}
\label{sec:properties_DG_scheme}

\begin{proposition}[Conservation]
	The scheme \eqref{esquema_DG_upw_Eyre_van_der_Zee} conserves the total mass of cells and nutrients in the following sense: for all $m\ge 0$,
	\begin{align*}
	\int_\Omega (u^{m+1}+n^{m+1})=\int_\Omega (u^m+n^m)\quad\text{and}\quad\int_\Omega (\up^{m+1}+n^{m+1})=\int_\Omega (\up^m+n^m).
	\end{align*}
\end{proposition}
\begin{proof}
	Just need to take $\overline{u}=1$ in \eqref{eq:esquema_DG_upw_Eyre_van_der_Zee_1} and $\overline{n}=1$ in \eqref{eq:esquema_DG_upw_Eyre_van_der_Zee_3} and add both expressions to obtain:
	\begin{align*}
		\int_\Omega (u^{m+1}+n^{m+1})=\int_\Omega (u^m+n^m).
	\end{align*}
	Moreover, due to the definition of the regularization $\Pi_1^h$, we have that $\int_\Omega u^{m+1}=\int_\Omega \up^{m+1}$ and $\int_\Omega u^{m}=\int_\Omega \up^{m}$, what yields
	\begin{align*}
		\int_\Omega (\up^{m+1}+n^{m+1})=\int_\Omega (\up^m+n^m).
	\end{align*}
\end{proof}

\begin{theorem}[Pointwise bounds]
	\label{thm:principio_del_maximo_DG_Van_der_Zee}
	
	Let $(u^{m+1}, \mu_u^{m+1}, n^{m+1})$ be a solution of the scheme \eqref{esquema_DG_upw_Eyre_van_der_Zee}, then $u^{m+1}, n^{m+1}\in[0,1]$ in $\Omega$ provided $u^{m},n^{m}\in[0,1]$ in $\Omega$.
\end{theorem}
\begin{proof}
	Firstly, we prove
	that $u^{m+1}, n^{m+1}\ge 0$ in $\Omega$.

	To prove that $u^{m+1}\ge 0$ we may take the following $\Pd_0(\T_h)$ test function
	\begin{align*}
		\overline{u}^*=
		\begin{cases}
			(\uKs^{m+1})_\ominus &\text{in }K^*\\
			0&\text{out of }K^*
		\end{cases},
	\end{align*}
	where $K^*$ is an element of $\T_h$ such that  $\uKs^{m+1}=\min_{K\in\T_h}u_{K}^{m+1}$. Then, by definition of $P(u,n)$ in \eqref{def:proliferation},
	$$
	\delta P_0\escalarL{P(u^{m+1},n^{m+1}) (\mu_n^{m+1}-\mu_u^{m+1})_\oplus }{\overline{u}^*}=0,
	$$
	equation \eqref{eq:esquema_DG_upw_Eyre_van_der_Zee_1} becomes
	\begin{equation}
		\label{test_u_0_vdz}
		\vert K^*\vert\delta_t \uKs^{m+1}(\uKs^{m+1})_\ominus = - C_u\aupw{\Pi_0\mu_u^{m+1}}{M(u^{m+1})_\oplus }{\bu^*}.
	\end{equation}
	
	Now,  since $\uL^{m+1}\ge \uKs^{m+1}$ we can assure that
	$$M^\uparrow(\uL^{m+1}) \ge M^\uparrow(\uKs^{m+1})\quad
	\hbox{and} \quad
	M^\downarrow(\uL^{m+1}) \le M^\downarrow(\uKs^{m+1}).
	$$
	Hence, using that the positive part is an increasing function, we obtain
	$$
	\aupw{\Pi_0\mu_u^{m+1}}{M(u^{m+1})}{\bu^*}\le 0,
	$$
	which yields $\vert K^*\vert\delta_t \uKs^{m+1}(\uKs^{m+1})_\ominus \ge0$.
	
	Consequently,
	$$
	0
	\le |K^*|(\delta_t \uKs^{m+1})(\uKs^{m+1})_\ominus 
	=
	-\frac{|K^*|}{\Delta t}\left((\uKs^{m+1})_\ominus ^2+\uKs^m(\uKs^{m+1})_\ominus \right)
	\le 0,
	$$
	which implies, since $\uKs^m\ge 0$, that $(\uKs^{m+1})_\ominus =0$. Hence $u^{m+1}\ge0$ in $\Omega$.
	
	Similarly, taking the following $\Pd_0(\T_h)$ test
	function in \eqref{eq:esquema_DG_upw_Eyre_van_der_Zee_3},
	\begin{align*}
		\overline{n}^*=
		\begin{cases}
			(\nKs^{m+1})_\ominus &\text{in }K^*\\
			0&\text{out of }K^*
		\end{cases}
	\end{align*}
	 where $K^*$ is an element of $\T_h$ such that  $\nKs^{m+1}=\min_{K\in\T_h}n_{K}^{m+1}$ we get $n^{m+1}\ge 0$ in $\Omega$.
	
	\
	
	Secondly, we prove that $u^{m+1},n^{m+1}\le 1$ in $\Omega$.
	
	To prove that $u^{m+1}\le 1$, taking the following test function in \eqref{eq:esquema_DG_upw_Eyre_van_der_Zee_1},
	\begin{align*}
		\overline{u}^*=
		\begin{cases}
			(\uKs^{m+1}-1)_\oplus &\text{in }K^*\\
			0&\text{out of }K^*
		\end{cases},
	\end{align*}
	where $K^*$ is an element of $\T_h$ such that  $\uKs^{m+1}=\max_{K\in\T_h}u_{K}^{m+1}$ and using similar arguments than above, we arrive at
	$$
	|K^*|\delta_t \uKs^{m+1}(\uKs^{m+1}-1)_\oplus \le0.
	$$
	Therefore, it is satisfied that
	\begin{align*}
		0&\ge
		|K^*|\delta_t \uKs^{m+1}(\uKs^{m+1}-1)_\oplus 
		=\frac{|K^*|}{\Delta t}\left((\uKs^{m+1}-1)+(1-\uKs^m)\right)(\uKs^{m+1}-1)_\oplus \\
		&=
		\frac{|K^*|}{\Delta t}\left((\uKs^{m+1}-1)^2_\oplus +(1-\uKs^m)(\uKs^{m+1}-1)_\oplus \right)
		\ge0,
	\end{align*}
	what yields $(\uKs^{m+1}-1)_\oplus =0$ and, therefore, $u^{m+1}\le1$ in $\Omega$.
	
	Finally, taking the test function in \eqref{eq:esquema_DG_upw_Eyre_van_der_Zee_3} 
		\begin{align*}
		\overline{n}^*=
		\begin{cases}
			(\nKs^{m+1}-1)_\oplus &\text{in }K^*\\
			0&\text{out of }K^*
		\end{cases}
	\end{align*}
	 where $K^*$ is an element of $\T_h$ such that  $\nKs^{m+1}=\max_{K\in\T_h}n_{K}^{m+1}$ we obtain, similarly, that $n^{m+1}\le 1$ in $\Omega$.
\end{proof}

The following result is a direct consequence of the previous Theorem \ref{thm:principio_del_maximo_DG_Van_der_Zee} and the equality \eqref{Pi_1} of the regularization $\Pi_1^h$.
\begin{corollary}
	\label{cor:principio_del_maximo_w_DG_Van_der_Zee}
	It satisfies $\up^{m+1}\in[0,1]$ in $\Omega$ provided $u^{m+1}\in[0,1]$ in $\Omega$.
\end{corollary}

\

\begin{theorem}[Energy law]
	\label{thm:energia_esquema_van_der_zee}
	Any solution
	of the scheme \eqref{esquema_DG_upw_Eyre_van_der_Zee} satisfies the following \textbf{discrete energy law} \begin{align}
		\label{ley_energia_discreta_van_der_zee}
		\delta_t E(\up^{m+1},n^{m+1})&+C_u\aupw{\mup_u^{m+1}}{M(u^{m+1})}{\Pi_0\mu_u^{m+1}}+C_n\aupw{\mu_n^{m+1}}{M(n^{m+1})}{\mu_n^{m+1}}\notag\\&
		+\frac{\Delta t \, \varepsilon^2}{2}\int_\Omega|\delta_t\nabla \up^{m+1}|^2
		+\frac{\Delta t}{2\delta}\int_\Omega \vert\delta_t n^{m+1}\vert^2\nonumber\\&
		+\delta P_0\int_\Omega P(u^{m+1},n^{m+1}) (\mu_n^{m+1}-\mup_u^{m+1})_\oplus ^2
		\le 0,
	\end{align}
	where the energy $E(\up,n)$ is defined in \eqref{energia_van_der_zee}.
\end{theorem}
\begin{proof}
	By taking
	$\overline{u}=\Pi_0\mu_u^{m+1}$,  $\overline{\mu}_u=\delta_t \up^{m+1}$, $\overline{n}=\mu_n^{m+1}$ in \eqref{eq:esquema_DG_upw_Eyre_van_der_Zee_1}--\eqref{eq:esquema_DG_upw_Eyre_van_der_Zee_3} and testing \eqref{eq:esquema_DG_upw_Eyre_van_der_Zee_4} by $\delta_t n^{m+1}$ we arrive at
	\begin{subequations}
		\label{esquema_DG_upw_Eyre_van_der_zee_energia}
		\begin{align}
			\label{eq:esquema_DG_upw_Eyre_van_der_zee_energia_1}
			\escalarL{\delta_t u^{m+1}}{\mup_u^{m+1}}&+C_u\aupw{\mup_u^{m+1}}{M(u^{m+1})}{\Pi_0\mu_u^{m+1}}&\notag\\&=\delta P_0 \escalarL{P(u^{m+1},n^{m+1}) (\mu_n^{m+1}-\mup_u^{m+1})_\oplus }{\mup_u^{m+1}},\\
			\label{eq:esquema_DG_upw_Eyre_van_der_zee_energia_2}
			\escalarL{\mu_u^{m+1}}{\delta_t \up^{m+1}}&=\varepsilon^2 \escalarLd{\nabla \up^{m+1}}{\delta_t\nabla \up^{m+1}}+ \escalarL{f(\up^{m+1},\up^m)}{\delta_t \up^{m+1}}\nonumber\\&\quad-\chi_0\escalarL{ n^{m+1}}{\delta_t \up^{m+1}},\\
			\label{eq:esquema_DG_upw_Eyre_van_der_zee_energia_3}
			\escalarL{\delta_t n^{m+1}}{\mu_n^{m+1}}&+C_n\aupw{\mu_n^{m+1}}{M(n^{m+1})}{\mu_n^{m+1}}&\notag\\&=-\delta P_0 \escalarL{P(u^{m+1},n^{m+1}) (\mu_n^{m+1}-\mup_u^{m+1})_\oplus }{\mu_n^{m+1}},\\
			\label{eq:esquema_DG_upw_Eyre_van_der_zee_energia_4}
			\escalarL{\mu_n^{m+1}}{\delta_t n^{m+1}}&=\frac{1}{\delta}\escalarL{n^{m+1}}{\delta_t n^{m+1}}  -\chi_0 \escalarL{\up^{m}}{\delta_t n^{m+1}},
		\end{align}
	\end{subequations}
	
	Observe that, by \eqref{eq:esquema_DG_Pi0}--\eqref{eq:esquema_DG_Pih1},
	\begin{align*}
		\escalarL{\delta_t \up^{m+1}}{\mu_u^{m+1}}&=\escalarL{\delta_t u^{m+1}}{\mu_u^{m+1}},\\
		\escalarL{\delta_t u^{m+1}}{\mu_u^{m+1}}&=\escalarL{\delta_t u^{m+1}}{\Pi_0 \mu_u^{m+1}},
	\end{align*}
	hence in particular 
	$$(\delta_t u^{m+1}, \Pi_0 \mu_u^{m+1})=(\delta_t \up^{m+1}, \mu_u^{m+1}).$$

	Then, by adding \eqref{eq:esquema_DG_upw_Eyre_van_der_zee_energia_1}--\eqref{eq:esquema_DG_upw_Eyre_van_der_zee_energia_4}, the previous terms cancel, remaining 
	\begin{align*}
		C_u\aupw{\mup_u^{m+1}}{&M(u^{m+1})}{\Pi_0\mu_u^{m+1}}
		+C_n\aupw{\mu_n^{m+1}}{M(n^{m+1})}{\mu_n^{m+1}}
		\\&+\varepsilon^2\escalarLd{\nabla \up^{m+1}}{\delta_t\nabla \up^{m+1}} +\escalarL{f(\up^{m+1},\up^m)}{\delta_t \up^{m+1}}\\
		&+\delta P_0\escalarL{P(u^{m+1},n^{m+1}) (\mu_n^{m+1}-\mup_u^{m+1})_\oplus }{\mu_n^{m+1}-\mup_u^{m+1}}\\&+\frac{1}{\delta}\escalarL{n^{m+1}}{\delta_t n^{m+1}}-\chi_0\escalarL{ n^{m+1}}{\delta_t \up^{m+1}}-\chi_0 \escalarL{\up^{m}}{\delta_t n^{m+1}}= 0 .
	\end{align*}
	Taking into account that
	\begin{align*}
	\varepsilon^2\escalarLd{\nabla \up^{m+1}}{\delta_t \nabla \up^{m+1}}
	&=\frac{\varepsilon^2}{2}\delta_t\int_\Omega |\nabla \up^{m+1}|^2 
	+\frac{\Delta t \,\varepsilon^2}{2}\int_\Omega|\delta_t\nabla \up^{m+1}|^2,\\
	\frac{1}{\delta}\escalarL{n^{m+1}}{\delta_t n^{m+1}}&=\frac{1}{2\delta}\delta_t\int_\Omega |n^{m+1}|^2 +\frac{\Delta t}{2\delta}\int_\Omega|\delta_t n^{m+1}|^2,\\
	\chi_0\delta_t\int_\Omega u^{m+1}n^{m+1}&=\chi_0\escalarL{ n^{m}}{\delta_t \up^{m+1}}+\chi_0 \escalarL{\up^{m+1}}{\delta_t n^{m+1}},\\
	\int_\Omega P(u^{m+1}, n^{m+1}) (\mu_n^{m+1}-\mup_u^{m+1})_\oplus ^2&=\escalarL{P(u^{m+1},n^{m+1}) (\mu_n^{m+1}-\mup_u^{m+1})_\oplus }{\mu_n^{m+1}-\mup_u^{m+1}},
	\end{align*}
	and by adding and subtracting $\delta_t\int_\Omega F(\up^{m+1})$,
	we get the following equality
	\begin{align*}
		\delta_t E(\up^{m+1},n^{m+1})&+C_u\aupw{\mup_u^{m+1}}{M(u^{m+1})}{\mup_u^{m+1}}+C_n\aupw{\mu_n^{m+1}}{M(n^{m+1})}{\mu_n^{m+1}}\notag
		\\&+\frac{\Delta t\,\varepsilon^2}{2}\int_\Omega|\delta_t\nabla \up^{m+1}|^2+\frac{\Delta t}{2\delta}\int_\Omega \vert\delta_t n^{m+1}\vert^2\notag\\&+\delta P_0\int_\Omega P(u^{m+1},n^{m+1}) (\mu_n^{m+1}-\mup_u^{m+1})_\oplus ^2\notag\\&=\delta_t\int_\Omega F(\up^{m+1})-\escalarL{f(\up^{m+1},\up^m)}{\delta_t \up^{m+1}}.
	\end{align*}

	Finally, from the convex-splitting approximation \eqref{convex-split} (see \cite{eyre_1998_unconditionally,guillen-gonzalez_linear_2013,acosta-soba_CH_2022}), one has  that
	$$
	\int_\Omega\delta_t F(\up^{m+1}) - \escalarL{f(\up^{m+1},\up^m)}{\delta_t( \up^{m+1})}\le0,
	$$
	which implies \eqref{ley_energia_discreta_van_der_zee}.

\end{proof}

\begin{corollary}
	The scheme \eqref{esquema_DG_upw_Eyre_van_der_Zee} is unconditionally energy stable in the sense
	$$
	E(\up^{m+1},n^{m+1})\le E(\up^{m},n^{m}),
	\quad \forall\, m\ge 0.
	$$
\end{corollary}
\begin{proof}
	It is straightforward to check (see \cite{acosta-soba_KS_2022}) that $$\aupw{\mup_u}{M(u^{m+1})}{\mup_u^{m+1}}\ge 0\quad\text{and}\quad\aupw{\mu_n^{m+1}}{M(n^{m+1})}{\mu_n^{m+1}}\ge 0.$$ Hence, using \eqref{ley_energia_discreta_van_der_zee} we conclude that $\delta_t E(\up^{m+1},n^{m+1})\le 0$.
\end{proof}

\

Now, we focus on the existence of the scheme \eqref{esquema_DG_upw_Eyre_van_der_Zee}. For this, we consider the following well-known result.
\begin{theorem}[Leray-Schauder fixed point theorem]
	\label{thm:Leray-Schauder}
	Let $\X$ be a Banach space and let $T\colon\X\longrightarrow\X$ be a continuous and compact operator. If the set $$\{x\in\X\colon x=\alpha \,T(x)\quad\text{for some } 0\le\alpha\le1\}$$ is bounded (with respect to $\alpha$), then $T$ has  at least one fixed point.
\end{theorem}

\begin{theorem}[Existence]
	\label{thm:existencia_solucion_Van_der_Zee_DG-UPW}
	There is at least one solution of the scheme
	\eqref{esquema_DG_upw_Eyre_van_der_Zee}.
\end{theorem}
\begin{proof}
	Given two functions $z_u, z_n\in\Pd_0(\T_h)$ with $0\le z_u, z_n \le 1$, we define the map
	$$T\colon \Pd_0\times \Pc_1\times\Pd_0\longrightarrow\Pd_0\times \Pc_1\times\Pd_0$$ such that $$T(\widehat{u},\widehat{\mu}_u,\widehat{n})=(u,\mu_u,n)\in\Pd_0(\T_h)\times\Pc_1(\T_h)\times\Pd_0(\T_h)$$ is the unique solution of the linear (and decoupled, computing first $\mu_n$, next $n$ and $u$, and finally $\mu_u$) scheme:
	\begin{subequations}
		\label{esquema_lineal_Leray-Schauder_DG_upw_Eyre_Van_der_Zee}
		\begin{align}
			\label{eq:esquema_lineal_Leray-Schauder_DG_upw_Eyre_Van_der_Zee_1}
			\frac{1}{\Delta t}\escalarL{u-z_u}{\overline{u}}&
			=-C_u\aupw{\Pi_0\widehat{\mu}}{M(\widehat{u})}{\overline{u}}+\delta P_0 \escalarL{P(\widehat{u},\widehat{n}) ( \mu_n-\Pi_0\widehat \mu_u)_\oplus }{\overline{u}},&&\forall\overline{u}\in \Pd_0(\T_h),\\
			\label{eq:esquema_lineal_Leray-Schauder_DG_upw_Eyre_Van_der_Zee_2}
			\escalarML{\mu_u}{\overline{\mu}_u}&=\varepsilon^2 \escalarL{\nabla \up}{\nabla\overline{\mu}_u}+ \escalarL{f(\up,\Pi_1^h z_u)}{\overline{\mu}_u}-\chi_0\escalarL{n}{\overline{\mu}_u},&&\forall\overline{\mu}_u\in \Pc_1(\T_h),\\
			\label{eq:esquema_lineal_Leray-Schauder_DG_upw_Eyre_Van_der_Zee_3}
			\frac{1}{\Delta t}\escalarL{ n-z_n}{\overline{n}}&=-C_n\aupw{\mu_n}{M(\widehat n)}{\overline{n}}-\delta P_0\escalarL{P(\widehat u, \widehat n) (\mu_n-\Pi_0\widehat\mu_u)_\oplus }{\overline{n}},&&\forall\overline{n}\in \Pd_0(\T_h),
		\end{align}
	\end{subequations}
	where
	\begin{equation}
		\label{eq:esquema_lineal_Leray-Schauder_DG_upw_Eyre_Van_der_Zee_4}
			\mu_n=\frac{1}{\delta}\widehat n -\chi_0 \Pi_0(\Pi_h^1 z_u).
	\end{equation}
	
	It is straightforward to check that, for any given $(\widehat{u},\widehat{\mu}_u,\widehat{n})\in \Pd_0(\T_h)\times \Pc_1(\T_h)\times \Pd_0(\T_h)$, there is a unique solution $(u,\mu_u,n)\in \Pd_0(\T_h)\times \Pc_1(\T_h)\times \Pd_0(\T_h)$. Therefore, the operator $T$ is well defined.

	Now we will prove that  $T$ is under the hypotheses of the Leray-Schauder fixed-point theorem \ref{thm:Leray-Schauder}.
	
	First, we check that $T$ is continuous. Let $\{(\widehat{u}_j,\widehat{\mu}_{u_j},\widehat{n}_j)\}_{j\in\N}\subset\Pd_0(\T_h)\times\Pc_1(\T_h)\times\Pd_0(\T_h)$ be a sequence such that $\lim_{j\to\infty}(\widehat{u}_j,\widehat{\mu}_{u_j},\widehat{n}_j)=(\widehat{u},\widehat{\mu}_u,\widehat{n})$. Taking into account that all norms are equivalent in $\Pd_0(\T_h)$ since it is a finite-dimensional space, the convergences $\widehat u_j\to \widehat u$ and $\widehat n_j\to \widehat n$ are equivalent to the convergences elementwise $ (\widehat u_j)_K\to \widehat n_K$ and $ (\widehat n_j)_K\to \widehat n_K$ for every $K\in\T_h$ (this may be seen, for instance, by using the norm $\norma{\cdot}_{L^\infty(\Omega)}$). Moreover, since $\Pi_0$ is continuous and $\Pi_0\widehat\mu_u\in\Pd_0(\T_h)$, the convergence $\Pi_0\widehat \mu_{u_j}\to  \Pi_0\widehat\mu_u$ is also equivalent to the convergence elementwise $ (\Pi_0\widehat\mu_{u_j})_K\to (\Pi_0\widehat\mu_{u})_K$ for every $K\in\T_h$. Finally, taking limits when $j\to \infty$ in \eqref{esquema_lineal_Leray-Schauder_DG_upw_Eyre_Van_der_Zee} (with $\widehat{u}\defeq\widehat{u}_j$, $\widehat{\mu}_u\defeq\widehat{\mu}_{u_j}$, $\widehat{n}\defeq\widehat{n}_j$ and $(u_j,\mu_{u_j},n_j)\defeq T(\widehat u_j,\widehat\mu_{u_j},\widehat n_j)$),
	and using the notion of convergence elementwise, we get that $$\lim_{j\to \infty} T(\widehat u_j,\widehat\mu_j, \widehat n_j)=T(\widehat u,\widehat\mu, \widehat n)=T\left(\lim_{j\to \infty}(\widehat u_j,\widehat\mu_j, \widehat n_j)\right),$$ hence $T$ is continuous. Therefore, $T$ is also compact since $\Pd_0(\T_h)$ and $\Pc_1(\T_h)$ have finite dimension.
	
	Finally, let us prove that the set 
	\begin{multline*}
		B=\{(u,\mu_u,n)\in\Pd_0\times\Pc_1\times\Pd_0\colon 
		(u,\mu_u,n)=\alpha\, T(u,\mu_u,n)\text{ for some } 0\le\alpha\le1\}
	\end{multline*}
	is bounded (independent of $\alpha$). The case $\alpha=0$ is trivial so we will assume that $\alpha\in(0,1]$.
	
	If $(u,\mu_u,n)\in B$, then $(u,\mu_u,n)\in\Pd_0(\T_h)\times \Pc_1(\T_h)\times\Pd_0(\T_h)$ is the solution of
	\begin{align}
		\label{eq:esquema_DG_upw_Eyre_Van_der_Zee_Leray-Schauder_1}
		\frac{1}{\Delta t}\escalarL{u-\alpha z_u}{\overline{u}}
		&=-\alpha \, C_u\,\aupw{\Pi_0\mu_u}{M({u})}{\overline{u}}\nonumber\\&\quad+\alpha\,\delta P_0 \escalarL{P(u,n) (\mu_n-\Pi_0\mu_u)_\oplus }{\overline{u}},
		&&\forall\overline{u}\in \Pd_0(\T_h),\\
		\label{eq:esquema_DG_upw_Eyre_Van_der_Zee_Leray-Schauder_2}
		\escalarML{\mu_u}{\overline{\mu}_u}&=\varepsilon^2 \escalarL{\nabla \up}{\nabla\overline{\mu}_u}+ \escalarL{f(\up,\Pi_1^hz_u)}{\overline{\mu}_u}-\chi_0\escalarL{n}{\overline{\mu}_u},&&\forall\overline{\mu}_u\in \Pc_1(\T_h),\\
		\label{eq:esquema_DG_upw_Eyre_Van_der_Zee_Leray-Schauder_3}
		\frac{1}{\Delta t}\escalarL{ n-\alpha z_n}{\overline{n}}&=-\alpha \, C_n\,\aupw{\mu_n}{M( n)}{\overline{n}}\nonumber\\&\quad-\alpha\,\delta P_0\escalarL{P(u,n) (\mu_n-\Pi_0\mu_u)_\oplus }{\overline{n}},&&\forall\overline{n}\in \Pd_0(\T_h)
	\end{align}
	where
	\begin{equation}
		\label{eq:esquema_DG_upw_Eyre_Van_der_Zee_Leray-Schauder_4}
		\mu_n=\frac{1}{\delta}n -\chi_0 \Pi_0(\Pi_h^1 z_u).
	\end{equation}
	Now, testing \eqref{eq:esquema_DG_upw_Eyre_Van_der_Zee_Leray-Schauder_1} by $\overline u=1$ and \eqref{eq:esquema_DG_upw_Eyre_Van_der_Zee_Leray-Schauder_3} by $\overline n=1$, we obtain $$\int_\Omega (u+n)=\alpha \int_\Omega (z_u+z_n).$$ Moreover, since $0\le z_u, z_n\le 1$, it can be proved that $0\le u, n\le 1$ using the same arguments than in Theorem \ref{thm:principio_del_maximo_DG_Van_der_Zee}. Therefore, we arrive at $$\norma{u}_{L^1(\Omega)}+\norma{n}_{L^1(\Omega)}\le \norma{z_u}_{L^1(\Omega)}+\norma{z_n}_{L^1(\Omega)}.$$
	Hence, using properties of $\up\in\Pc_1(\T_h)$, since $0\le u\le1$ then $0\le \up\le1$, and 
	$$
	\norma{\up}_{L^1(\Omega)}=\norma{u}_{L^1(\Omega)}\le \norma{z_u}_{L^1(\Omega)}+\norma{z_n}_{L^1(\Omega)}.
	$$
	
	Now, we will check that $\mu_u$ is bounded. Testing 
	\eqref{eq:esquema_DG_upw_Eyre_Van_der_Zee_Leray-Schauder_2} with $\overline{\mu}_u=\mu_u$ we obtain that
	$$
		\normaL{\mu_u}^2
\le \varepsilon^2\norma{\up}_{H^1(\Omega)}\norma{\mu_u}_{H^1(\Omega)}+\normaL{f(\up,\Pi_1^hz_u)}\normaL{\mu_u}+\normaL{n}\normaL{\mu_u}.
$$
	The norms are equivalent in the finite-dimensional space $\Pc_1(\T_h)$, therefore, there are $K_1,K_2\ge 0$ such that
	$$\normaL{\mu_u}\le \varepsilon^2K_1\norma{\up}_{L^1(\Omega)}+\normaL{f(\up,\Pi_1^hz_u)}+K_2\norma{n}_{L^1(\Omega)}.$$
	Consequently, since $\normaL{f(\up,\Pi_1^hz_u)}$ is bounded due to $0\le \up,\Pi_1^h z_u\le 1$ and $\norma{\up}_{L^1(\Omega)}$ and $\norma{n}_{L^1(\Omega)}$ are also bounded, we conclude that $\normaL{\mu_u}$ is bounded.

	Since $\Pd_0(\T_h)$ and $\Pc_1(\T_h)$ are finite-dimensional spaces where all the norms are equivalent, we have proved that $B$ is bounded.
	
	Thus, using the Leray-Schauder fixed point theorem \ref{thm:Leray-Schauder}, there is a solution $(u,\mu_u,n)$ of the scheme
	\eqref{esquema_DG_upw_Eyre_van_der_Zee}.
\end{proof}

\section{Numerical experiments}
\label{sec:numer-experiments}

Now, we will present several numerical experiments that match the results presented in the previous section. We assume that $\Omega=[-10,10]^2$, $\varepsilon=0.1$, $\delta=0.01$ and we consider the mesh is shown in Figure~\ref{fig:mesh} which satisfies the Hypothesis~\ref{hyp:mesh}. The nonlinear coupled scheme \eqref{esquema_DG_upw_Eyre_van_der_Zee} is approximated by Newton's method.
\begin{figure}
	\centering
	\includegraphics[scale=0.24]{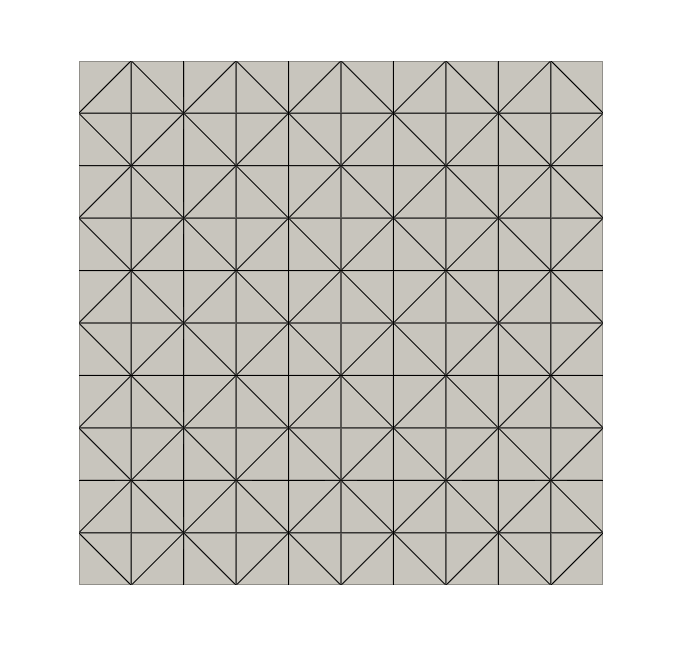}

	\caption{Mesh used for domain discretization.}
	\label{fig:mesh}
\end{figure}

These results have been computed using the Python interface of the library \texttt{FEniCSx}, \cite{AlnaesEtal2014, ScroggsEtal2022,BasixJoss}, and the figures have been plotted using \texttt{PyVista}, \cite{sullivan2019pyvista}.

Notice that, as mentioned in Subsection \ref{sec:fully_discrete_scheme}, $\up^{m}$ is considered the approximation of the phase filed variable $u$ by the scheme \eqref{esquema_DG_upw_Eyre_van_der_Zee}. Therefore, all the results shown in this section correspond with this approximation. On the other hand, although $n^m$ is taken as the approximation of the nutrients variable $n$, for the ease of visualization, $\Pi_1^h n^m$ has been plotted in Figures~\ref{fig:test-1_initial_cond}, \ref{fig:test-1_1}, \ref{fig:test-1_2}, \ref{fig:test-2_initial_cond} and \ref{fig:test-2_reference}. 

\subsection{Three tumors aggregation}
\label{sec:numer-experiments_1}
We define the following initial conditions which are of the same type than those in \cite{wu_stabilized_2014}:
\begin{align*}
	u_0 &= \frac{1}{2}\left[\tanh\left(\frac{1 - \sqrt{(x- 2)^2 + (y - 2)^2}}{\sqrt{2}\varepsilon}\right)
	+ \tanh\left(\frac{1 - \sqrt{(x - 3)^2 + (y + 5)^2}}{\sqrt{2}\varepsilon}\right)\right.\\&\quad\left.
	+ \tanh\left(\frac{1.73 - \sqrt{(x + 1.5)^2 + (y + 1.5)^2}}{\sqrt{2}\varepsilon}\right) + 3\right],\\
	n_0 &= 1.0 - u_0.
\end{align*}
These initial conditions are shown in Figure \ref{fig:test-1_initial_cond}. As one may observe, we assume that, at the beginning, the nutrients are fully consumed in the area occupied by the initial tumor.

\begin{figure}
	\centering
	\begin{tabular}{cc}
		\hspace*{-1.1cm}$\boldsymbol{u_0}$ & \hspace*{-1.1cm}$\boldsymbol{n_0}$\\
		\includegraphics[scale=0.24]{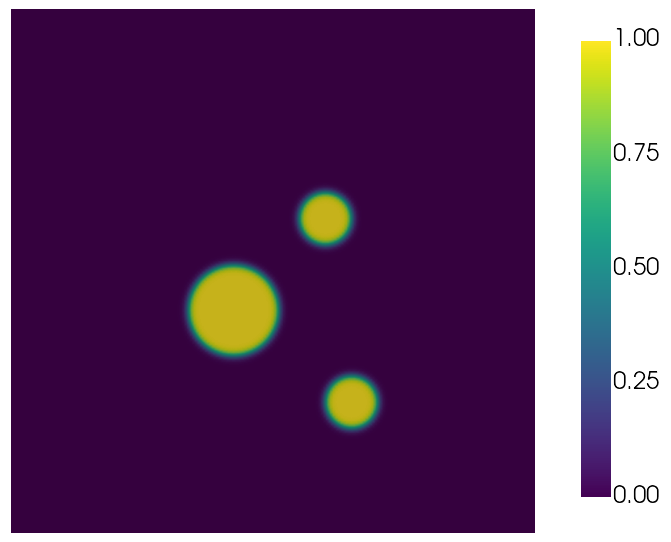} &
		\includegraphics[scale=0.24]{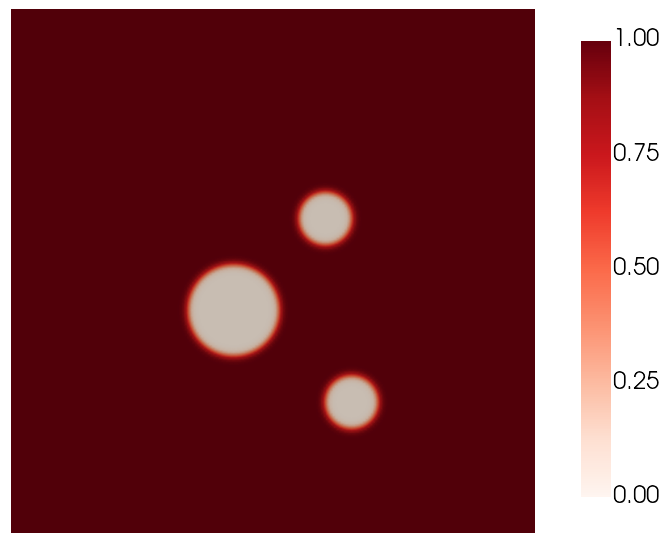}
	\end{tabular}
	\caption{Initial conditions for test \ref{sec:numer-experiments_1} ($u_0$ left, $n_0$ right).}
	\label{fig:test-1_initial_cond}
\end{figure}

Moreover, we set $C_u=100$, $C_n=100\cdot 10^{-4}$, $P_0=125$  $h\approx 0.14$ and we use the following symmetric mobility and proliferation functions:
\begin{equation}
	\label{symmetric_functions}
	M(v)=h_{1,1}(v),\quad P(u,n)=h_{1,1}(u)n_\oplus.
\end{equation}
We are going to compare the upwind DG scheme \eqref{esquema_DG_upw_Eyre_van_der_Zee} and the $\Pc_1(\T_h)$-FE approximation of the time discrete scheme \eqref{esquema:time_discrete}. We consider two different cases: $\chi_0=0$ and $\chi_0=10$, i.e. without and with cross-diffusion, respectively.

On the one hand, the experiment without cross-diffusion ($\chi_0=0$ and $\Delta t=10^{-5}$) is plotted in Figure~\ref{fig:test-1_1}. As one may notice, both schemes provide a similar approximation. The approximations preserve, approximately in the case of FE, the pointwise bounds of the variables $u$ and $n$ and the energy stability, see Figures~\ref{fig:test-1_1_bounds} and \ref{fig:test-1_energy} (left).

\begin{figure}
	\centering
	\begin{tabular}{ccccc}
		& &\hspace*{-1cm}$t=2.5\cdot 10^{-2}$ & \hspace*{-1cm}$t=5.0\cdot 10^{-2}$ & \hspace*{-1cm}$t=7.5\cdot 10^{-2}$ \\
		\multirow{3}{*}{\vspace*{-2.7cm}$\boldsymbol{u}$}&\rotatebox[origin=c]{90}{\textbf{DG}} &
		\raisebox{-0.47\height}{\includegraphics[scale=0.204]{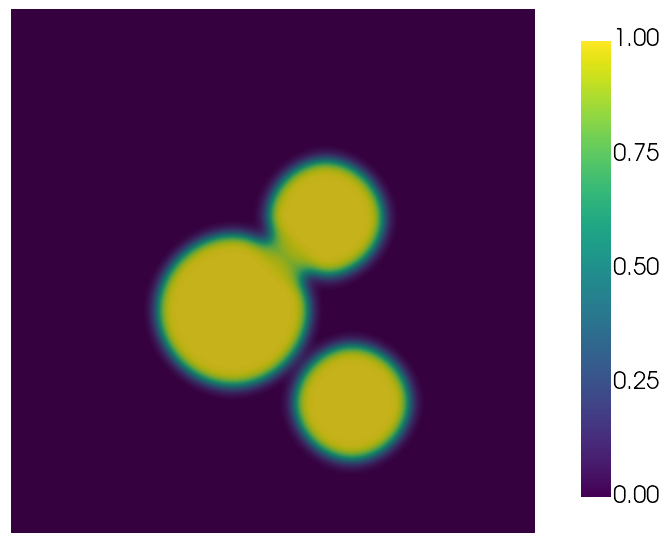}} &
		\raisebox{-0.47\height}{\includegraphics[scale=0.204]{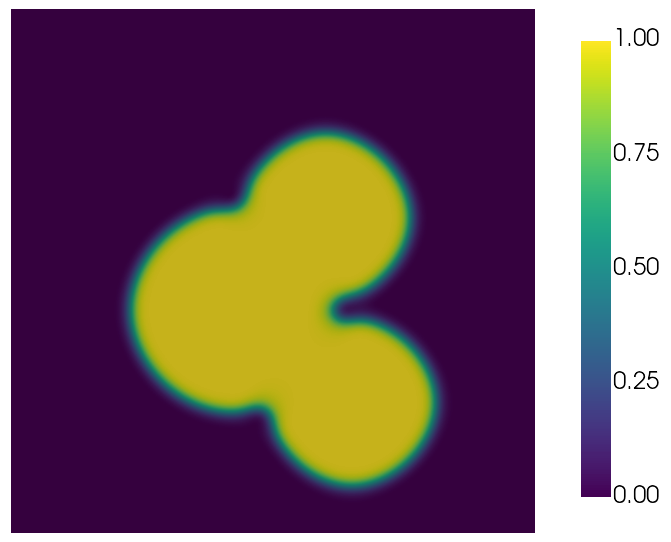}} &
		\raisebox{-0.47\height}{\includegraphics[scale=0.204]{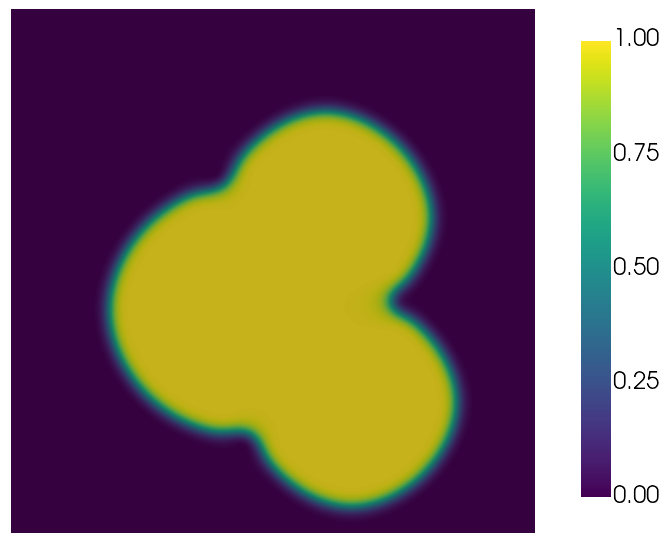}} \\
		&\rotatebox[origin=c]{90}{\textbf{FE}} &
		\raisebox{-0.47\height}{\includegraphics[scale=0.204]{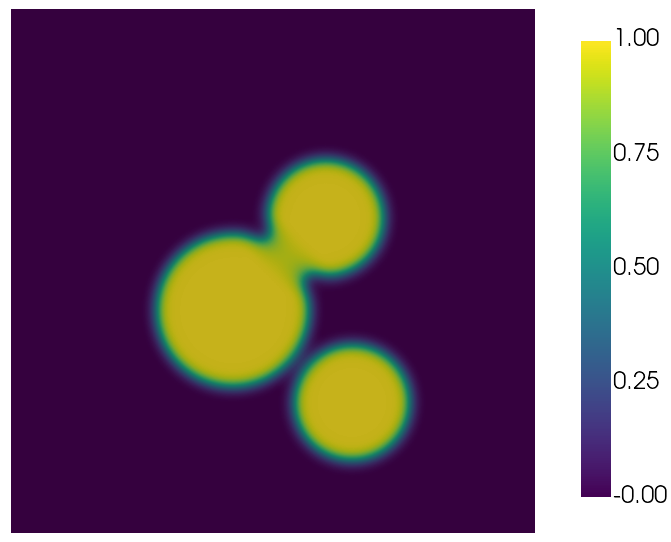}} &
		\raisebox{-0.47\height}{\includegraphics[scale=0.204]{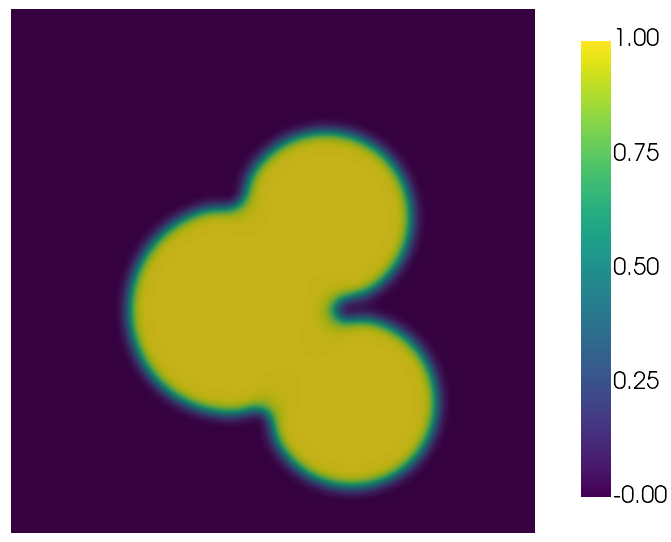}} &
		\raisebox{-0.47\height}{\includegraphics[scale=0.204]{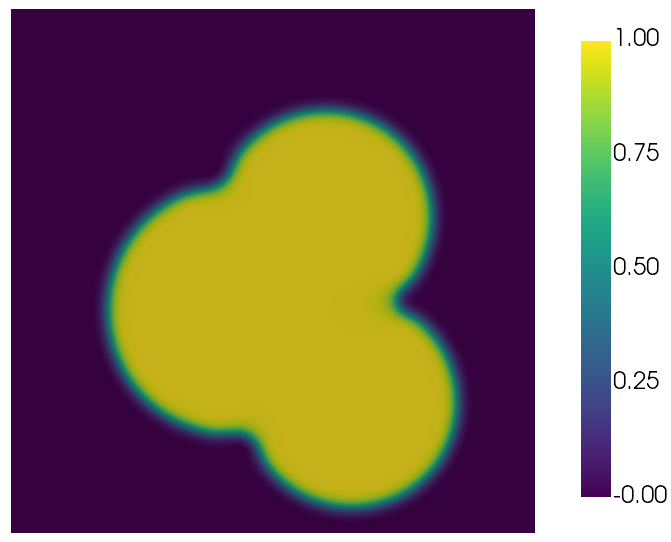}}\\
		\multirow{3}{*}{\vspace*{-2.7cm}$\boldsymbol{n}$}&\rotatebox[origin=c]{90}{\textbf{DG}} &
		\raisebox{-0.47\height}{\includegraphics[scale=0.204]{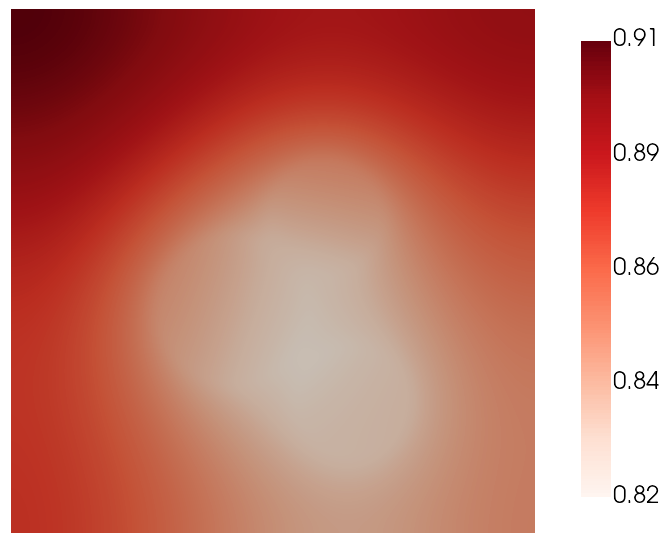}} &
		\raisebox{-0.47\height}{\includegraphics[scale=0.204]{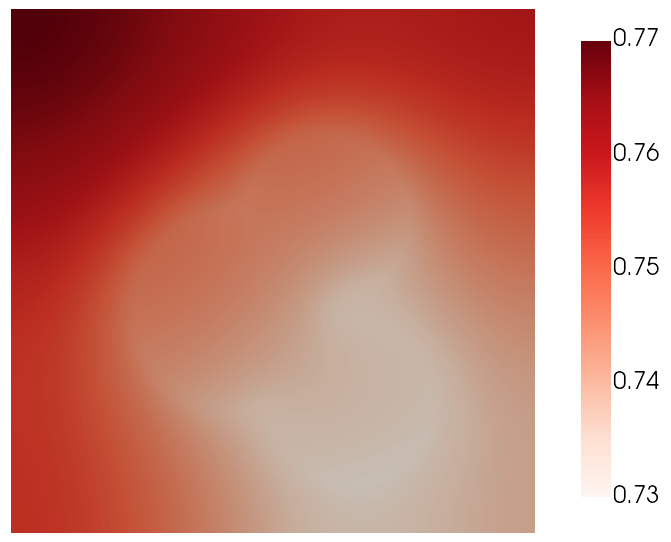}} &
		\raisebox{-0.47\height}{\includegraphics[scale=0.204]{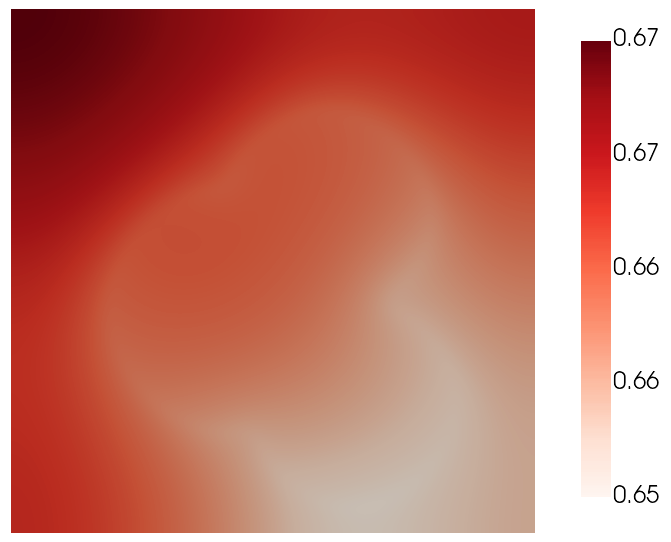}} \\
		&\rotatebox[origin=c]{90}{\textbf{FE}} &
		\raisebox{-0.47\height}{\includegraphics[scale=0.204]{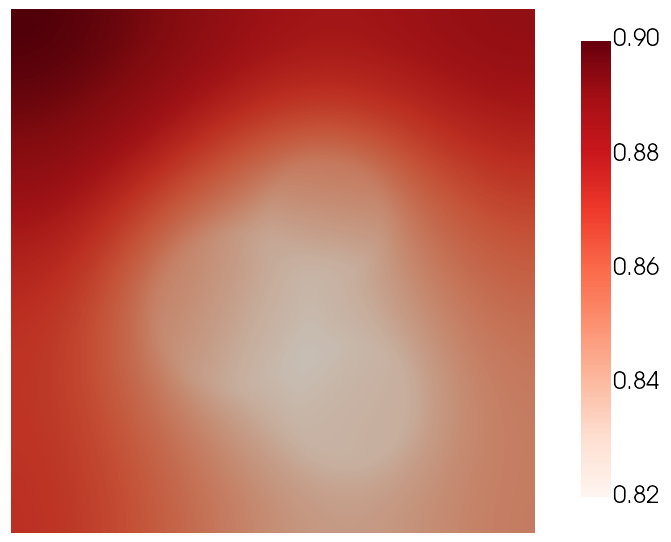}} &
		\raisebox{-0.47\height}{\includegraphics[scale=0.204]{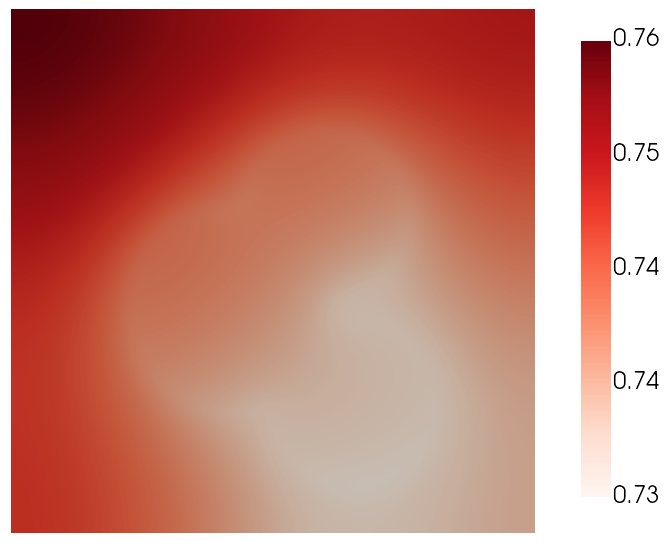}} &
		\raisebox{-0.47\height}{\includegraphics[scale=0.204]{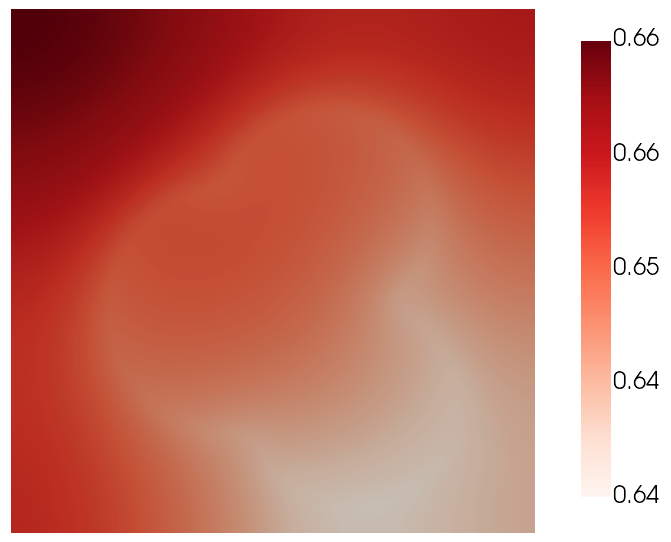}}
	\end{tabular}
	\caption{Tumor and nutrients for test \eqref{sec:numer-experiments_1} with $\chi_0=0$ at different time steps.}
	\label{fig:test-1_1}
\end{figure}

On the other hand, the test with cross-diffusion ($\chi_0=10$ and $\Delta t=5\cdot 10^{-6}$) is plotted in Figures~\ref{fig:test-1_2}. In this case, one may notice that, while DG scheme provides a good approximation of the solution, FE solution shows a lot of spurious oscillations. These numerical instabilities lead to a loss of the maximum principle while it is preserved by the DG scheme, see Figure~\ref{fig:test-1_2_bounds}. In both cases, the schemes preserve the energy stability of the model as expected, see Figure~\ref{fig:test-1_energy} (right).

\begin{figure}
	\centering
	\begin{tabular}{ccccc}
		& & \hspace*{-1cm}$t=7.5\cdot 10^{-3}$ & \hspace*{-1cm}$t=1.5\cdot 10^{-2}$ & \hspace*{-1cm}$t=4\cdot 10^{-2}$ \\
		\multirow{3}{*}{\vspace*{-2.7cm}$\boldsymbol{u}$}&\rotatebox[origin=c]{90}{\textbf{DG}} &
		\raisebox{-0.47\height}{\includegraphics[scale=0.204]{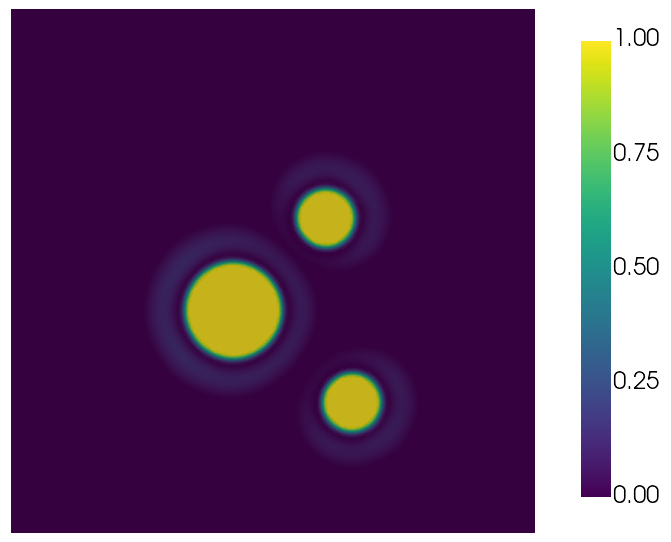}} &
		\raisebox{-0.47\height}{\includegraphics[scale=0.204]{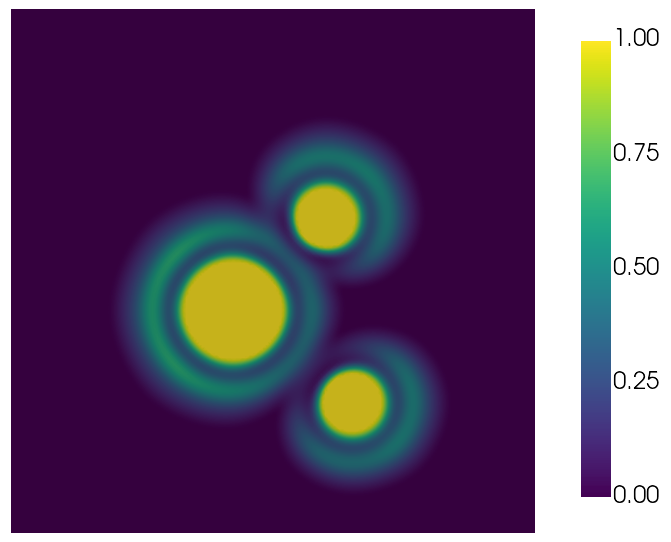}} &
		\raisebox{-0.47\height}{\includegraphics[scale=0.204]{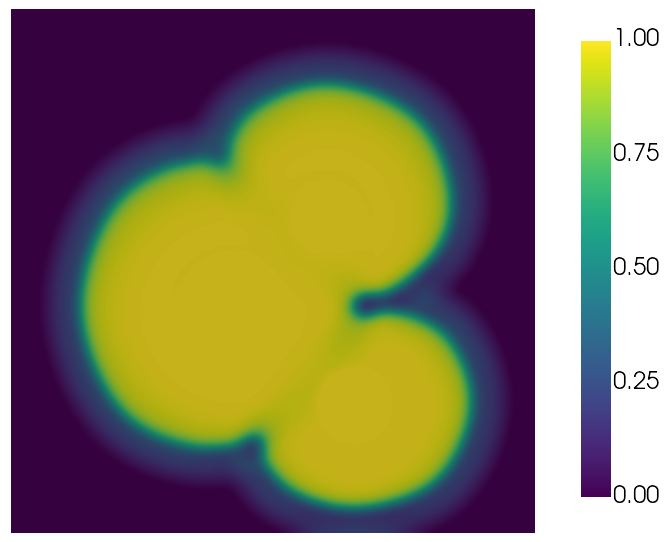}} \\
		&\rotatebox[origin=c]{90}{\textbf{FE}} &
		\raisebox{-0.47\height}{\includegraphics[scale=0.204]{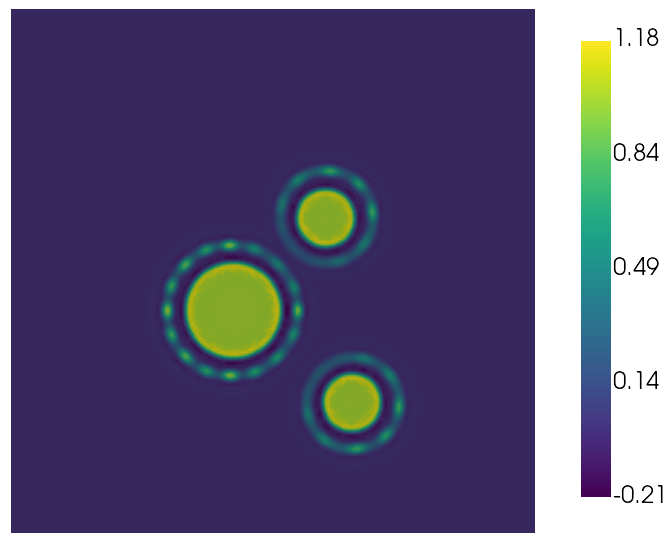}} &
		\raisebox{-0.47\height}{\includegraphics[scale=0.204]{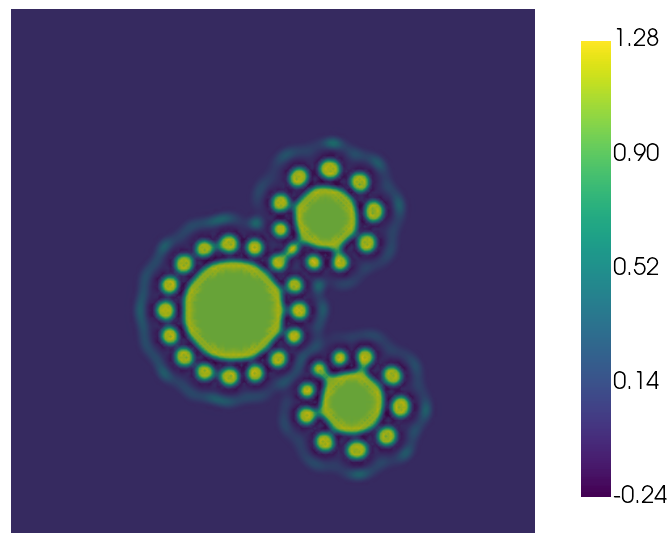}} &
		\raisebox{-0.47\height}{\includegraphics[scale=0.204]{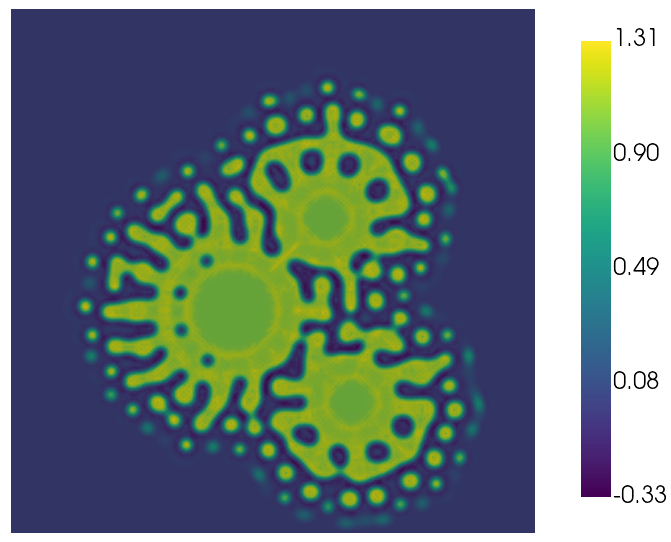}}\\
		\multirow{3}{*}{\vspace*{-2.7cm}$\boldsymbol{n}$}&\rotatebox[origin=c]{90}{\textbf{DG}} &
		\raisebox{-0.47\height}{\includegraphics[scale=0.204]{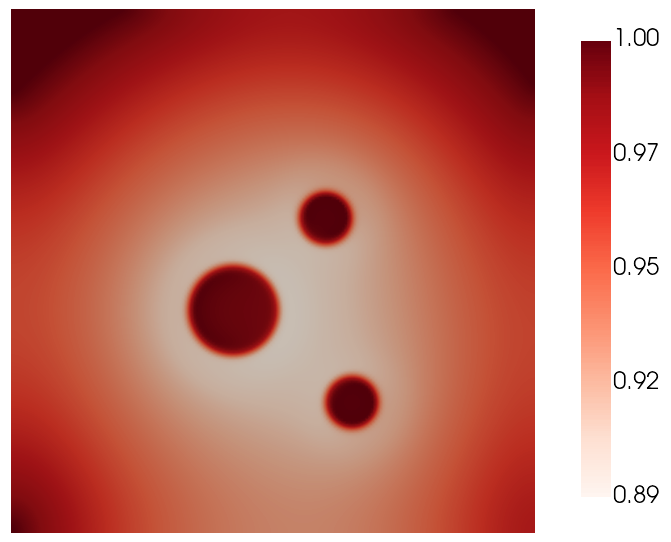}} &
		\raisebox{-0.47\height}{\includegraphics[scale=0.204]{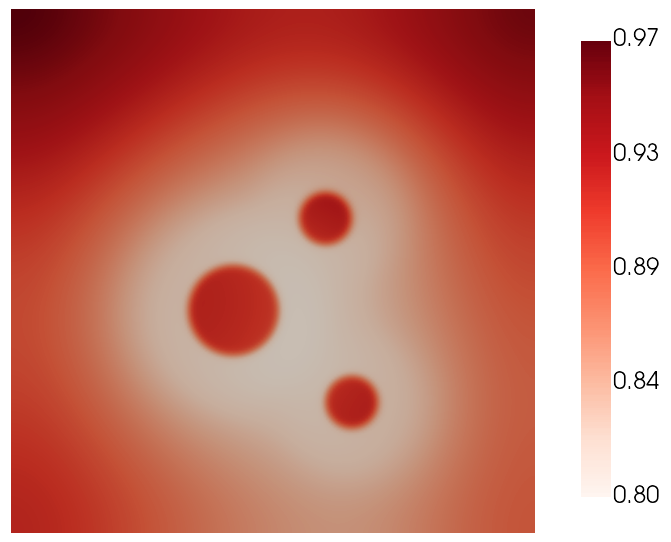}} &
		\raisebox{-0.47\height}{\includegraphics[scale=0.204]{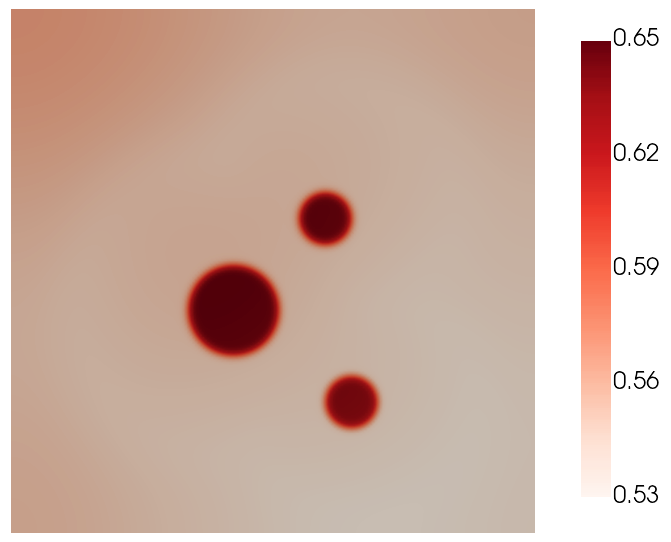}} \\
		&\rotatebox[origin=c]{90}{\textbf{FE}} &
		\raisebox{-0.47\height}{\includegraphics[scale=0.204]{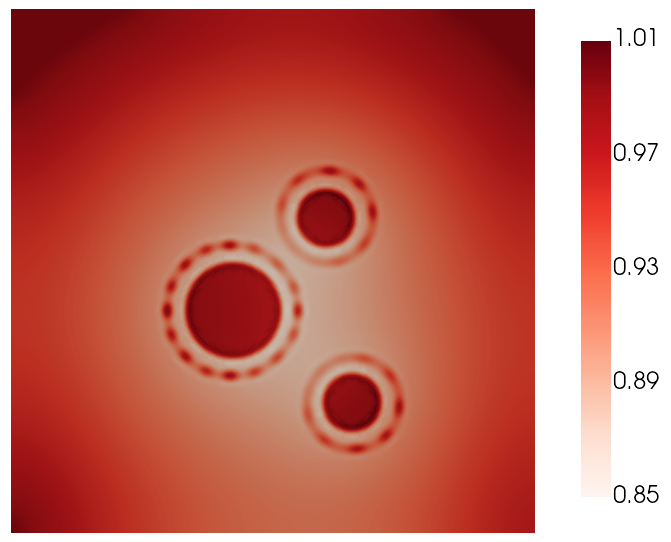}} &
		\raisebox{-0.47\height}{\includegraphics[scale=0.204]{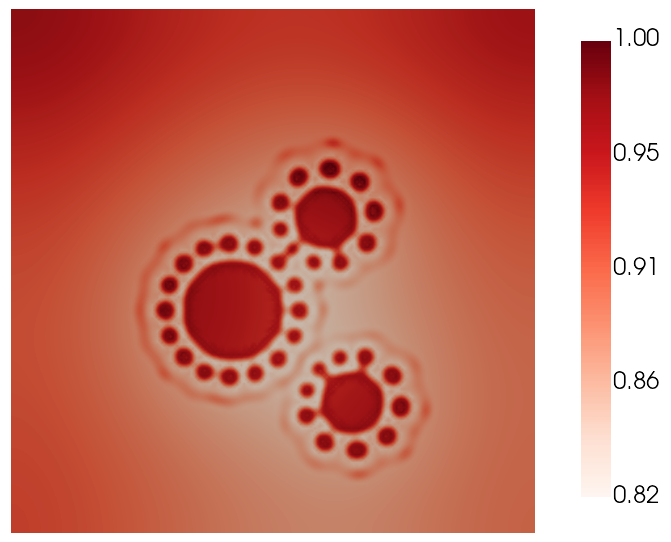}} &
		\raisebox{-0.47\height}{\includegraphics[scale=0.204]{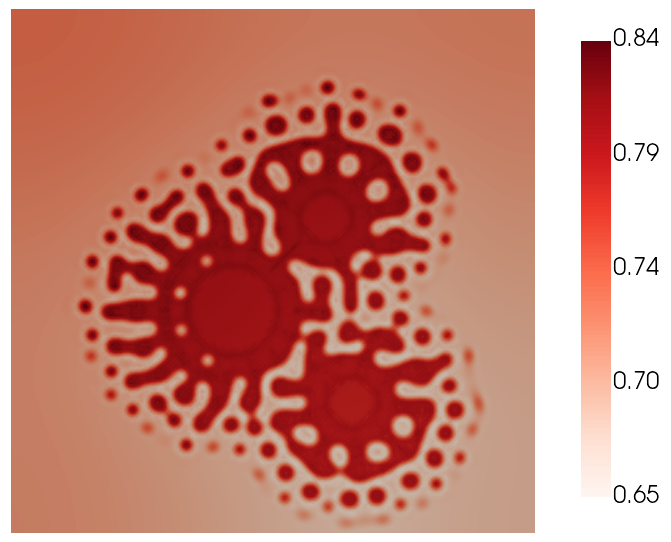}}
	\end{tabular}
	\caption{Tumor and nutrients for test \eqref{sec:numer-experiments_1} with $\chi_0=10$ at different time steps.}
	\label{fig:test-1_2}
\end{figure}

\begin{figure}
	\centering
	\begin{tabular}{cc}
		\hspace*{1cm}$\boldsymbol{u}$ & \hspace*{0.5cm}$\boldsymbol{n}$\\
		\includegraphics[scale=0.45]{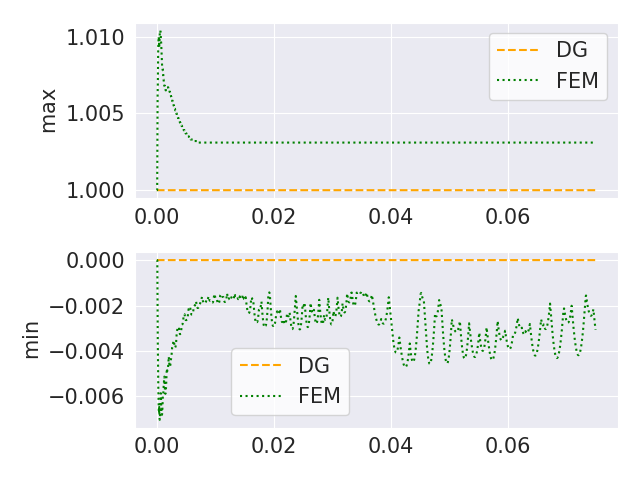} &
		\includegraphics[scale=0.45]{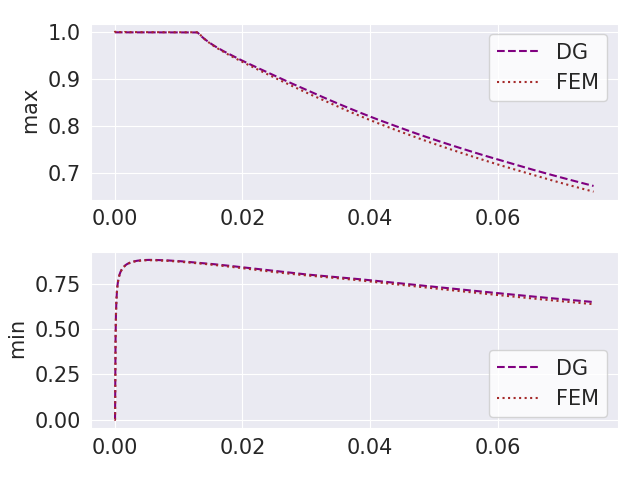}
	\end{tabular}
	\caption{Pointwise bounds of the approximations for test \ref{sec:numer-experiments_1} with $\chi_0=0$ ($u$ left, $n$ right).}
	\label{fig:test-1_1_bounds}
\end{figure}
\begin{figure}
	\centering
	\begin{tabular}{cc}
		\hspace*{1cm}$\boldsymbol{u}$ & \hspace*{0.5cm}$\boldsymbol{n}$\\
		\includegraphics[scale=0.45]{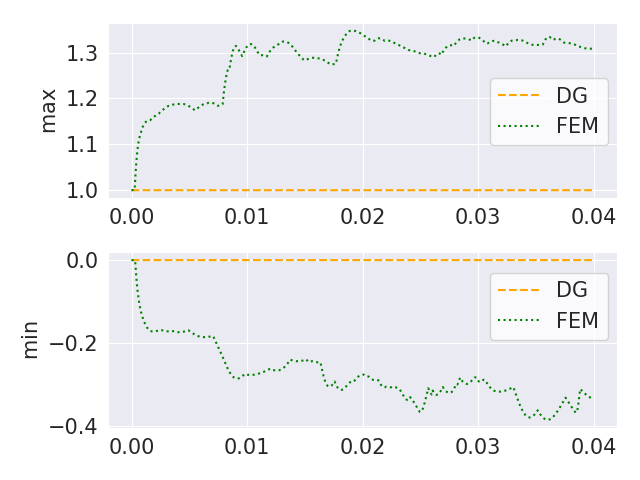} &
		\includegraphics[scale=0.45]{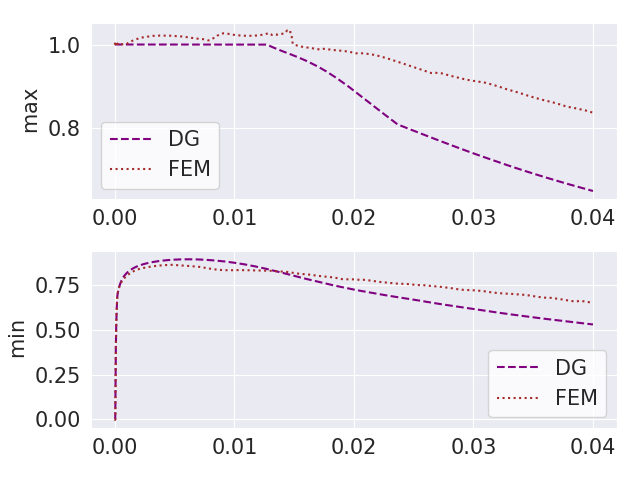}
	\end{tabular}
	\caption{Pointwise bounds of the approximations for test \ref{sec:numer-experiments_1} with $\chi_0=10$ ($u$ left, $n$ right).}
	\label{fig:test-1_2_bounds}
\end{figure}
\begin{figure}
	\centering
	\hspace*{0.8cm}\textbf{Energy}

	\begin{tabular}{cc}
		\includegraphics[scale=0.45]{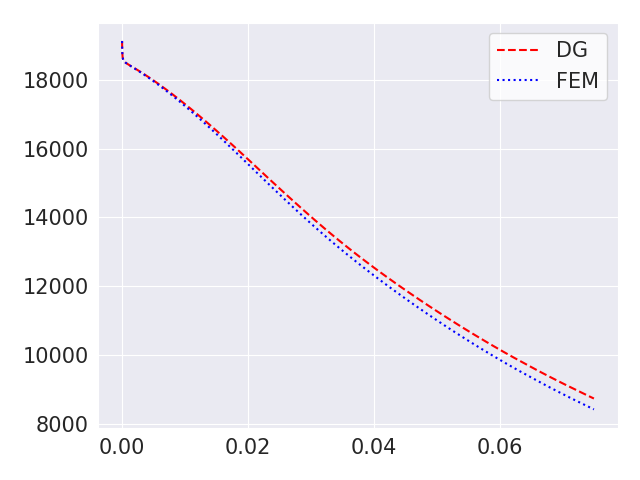} &
		\includegraphics[scale=0.45]{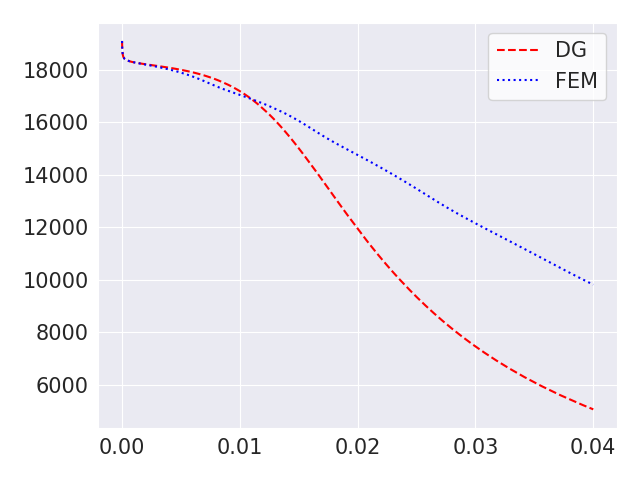}
	\end{tabular}
	\caption{$E(\up^m,n^m)$ for test \ref{sec:numer-experiments_1} with $\chi_0=0$ (left) and with $\chi_0=10$ (right).}
	\label{fig:test-1_energy}
\end{figure}

Furthermore, it is remarkable to emphasize that the convergence of Newton's method for the FE scheme requires a very small time step. In this sense, the previous tests where shown for a small enough time step so that Newton's method converges for both schemes. Conversely, the upwind DG scheme \eqref{esquema_DG_upw_Eyre_van_der_Zee} does converge for larger time steps. In practice, we have been able to compute the approximation given by the DG scheme for this test with time steps up to $\Delta t=10^{-4}$.

\subsection{Irregular tumor growth}
\label{sec:numer-experiments_2}

In this test, we show the irregular growth of a tumor due to the irregular distribution of the nutrients over the domain. It is important to notice the well behavior of the scheme \eqref{esquema_DG_upw_Eyre_van_der_Zee} which allow us to capture different irregular growth processes even in the cases with important cross-diffusion in which we cannot expect FE to work as shown in Subsection \ref{sec:numer-experiments_1}.

 In particular, we consider the following initial conditions for tumor cells and nutrients:
\begin{align*}
	u_0 &= \frac{1}{2}\left[\tanh\left(\frac{1.75 - \sqrt{x^2 + y^2}}{\sqrt{2}\varepsilon}\right) + 1\right],\\
	n_0 &= \frac{1}{2}(1-u_0)
	+ \frac{1}{4}\left[\tanh\left(\frac{1 - \sqrt{(x - 2.45)^2 + (y - 1.45)^2}}{\sqrt{2}\varepsilon}\right)\right.\\
	&\quad
	\left.+ \tanh\left(\frac{1.75 - \sqrt{(x + 3.75)^2 + (y - 1)^2}}{\sqrt{2}\varepsilon}\right)
	+ \tanh\left(\frac{2.5 - \sqrt{x^2 + (y + 5)^2}}{\sqrt{2}\varepsilon}\right) + 3\right],
\end{align*}
which are shown in Figure \ref{fig:test-2_initial_cond}.

\begin{figure}
	\centering
	\begin{tabular}{cc}
		\hspace*{-1.1cm}$\boldsymbol{u_0}$ & \hspace*{-1.1cm}$\boldsymbol{n_0}$\\
		\includegraphics[scale=0.24]{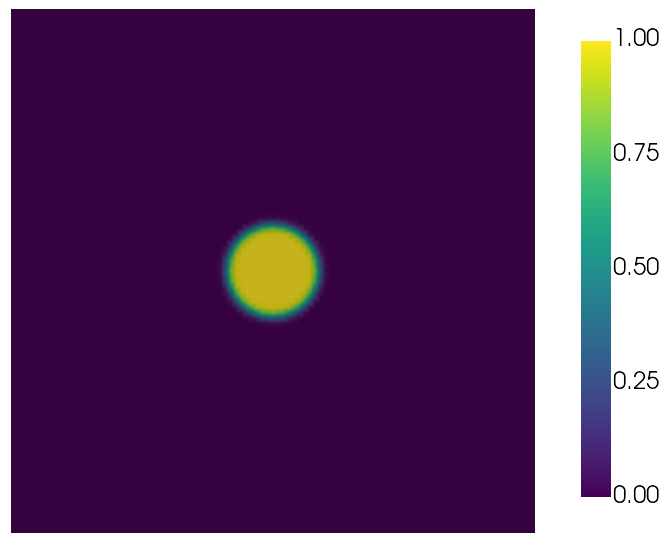} &
		\includegraphics[scale=0.24]{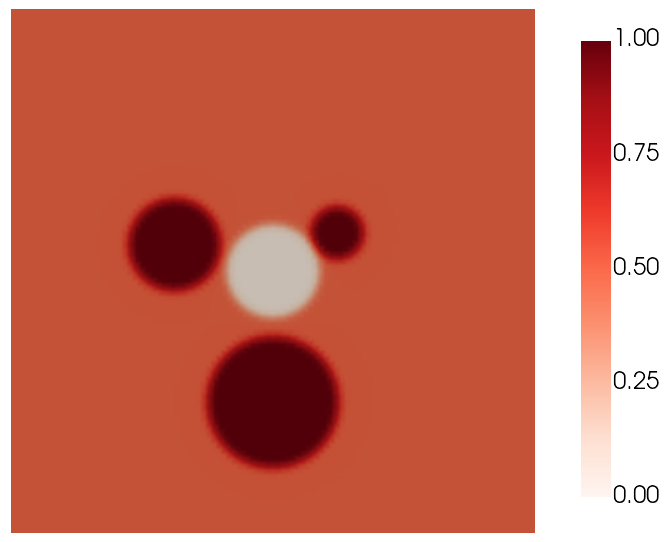}
	\end{tabular}
	\caption{Initial conditions for test \ref{sec:numer-experiments_2} ($u_0$ left, $n_0$ right).}
	\label{fig:test-2_initial_cond}
\end{figure}

We represent the behavior of the solution of the model under different set of parameters, see Figures~\ref{fig:test-2_reference}--\ref{fig:test-2_chi-1_u}. We set $C_u=2.8$, $C_n=2.8\cdot 10^{-4}$, $h\approx 0.28$ for every experiment and we vary the rest of the parameters with respect to the reference test in Figure ~\ref{fig:test-2_reference} ($P_0=0.5$, $\chi_0=0.1$ and $\Delta t=0.1$). For the sake of brevity, we only show the nutrients variable for the reference test.

In fact, we have considered two different types of mobility and proliferation functions. On the one hand, the typical symmetric functions used in the previous experiment \eqref{symmetric_functions} have been used (see the top rows of Figures~\ref{fig:test-2_reference}--\ref{fig:test-2_chi-1_u}). However, on the other hand, we have considered the following non-symmetric choice of the mobility and proliferation functions
\begin{equation}
	\label{nonsymmetric_functions}
	M(v)=h_{5,1}(v),\quad P(u,n)=h_{1,3}(u) n_\oplus,
\end{equation}
whose associated results are plotted in the bottom row of Figures~\ref{fig:test-2_reference}--\ref{fig:test-2_chi-1_u}.

	The proliferation function in \eqref{nonsymmetric_functions} has been chosen to model a very quick tumor growth and nutrient consumption at the non-saturated state ($u\simeq 0$) that decays until the tumor is fully saturated  ($u\simeq 1$). Moreover, the choice of the mobility function in \eqref{nonsymmetric_functions} is thought to prevent the dissemination of the tumor and the nutrients in a non-saturated state ($u,n\simeq 0$) leading to a more local tumor/nutrient interaction due to the proliferation term.

	Of course, the choice of these functions does not limit to those in \eqref{nonsymmetric_functions} and other degenerated mobility and proliferation functions can be considered. In this sense, we would like to emphasize that the choice of these functions may be motivated by different types of tumor which might show particular growth and interaction with nutrients behaviors.

Indeed, we can observe the different expected behaviors of the solution for both choices of mobilities and proliferation functions in Figures~\ref{fig:test-2_reference}--\ref{fig:test-2_chi-1_u}. On the one hand, we may notice a local growth of the tumor where a proliferation area appears around the fully saturated tumor due to \eqref{nonsymmetric_functions}. Conversely, we can observe an eventual dissemination of the tumor all over the domain using \eqref{symmetric_functions} in the cases where the proliferation term is more significant than the cross-diffusion allowing the tumor to grow by consuming nutrients.

\begin{figure}
	\centering
	\begin{tabular}{ccccc}
		& & \hspace*{-1cm}$t=10.0$ & \hspace*{-1cm}$t=20.0$ & \hspace*{-1cm}$t=50.0$ \\
		\multirow{3}{*}{\vspace*{-2.7cm}$\boldsymbol{u}$}&\rotatebox[origin=c]{90}{\textbf{Symmetric \eqref{symmetric_functions}}} &
		\raisebox{-0.47\height}{\includegraphics[scale=0.204]{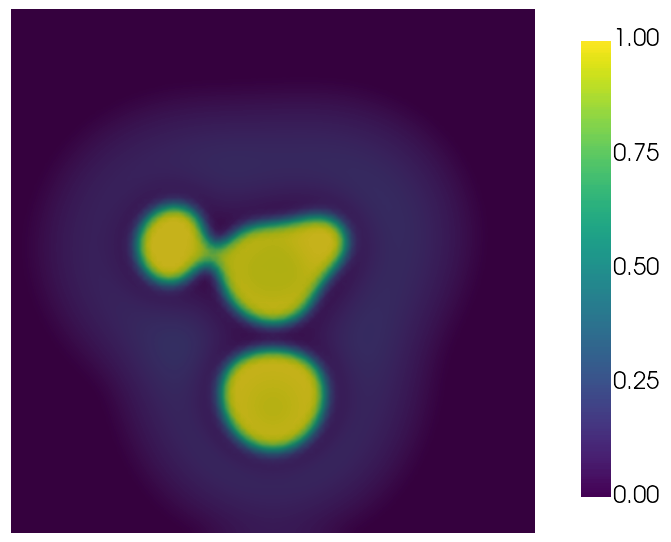}} &
		\raisebox{-0.47\height}{\includegraphics[scale=0.204]{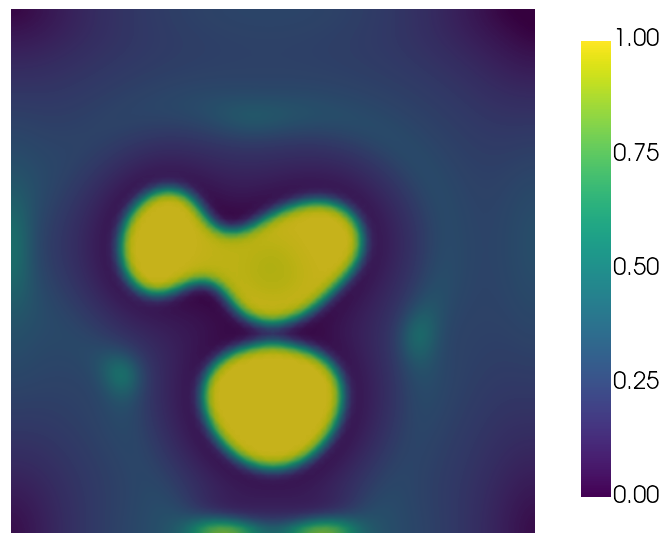}} &
		\raisebox{-0.47\height}{\includegraphics[scale=0.204]{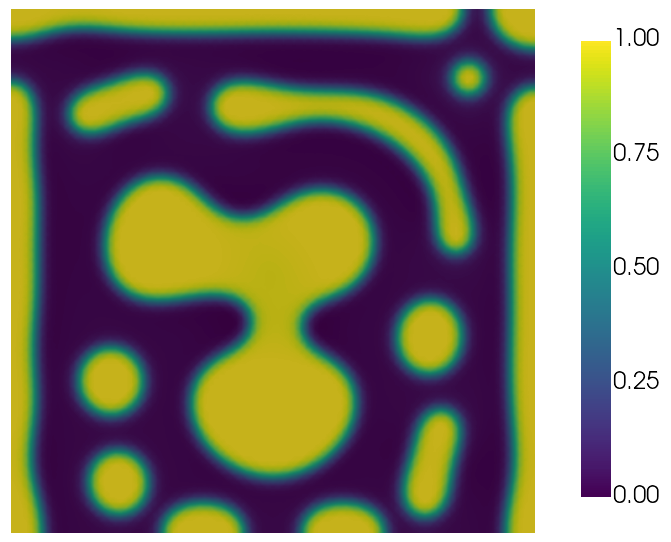}} \\
		&\rotatebox[origin=c]{90}{\textbf{Non-symmetric \eqref{nonsymmetric_functions}}} &
		\raisebox{-0.47\height}{\includegraphics[scale=0.204]{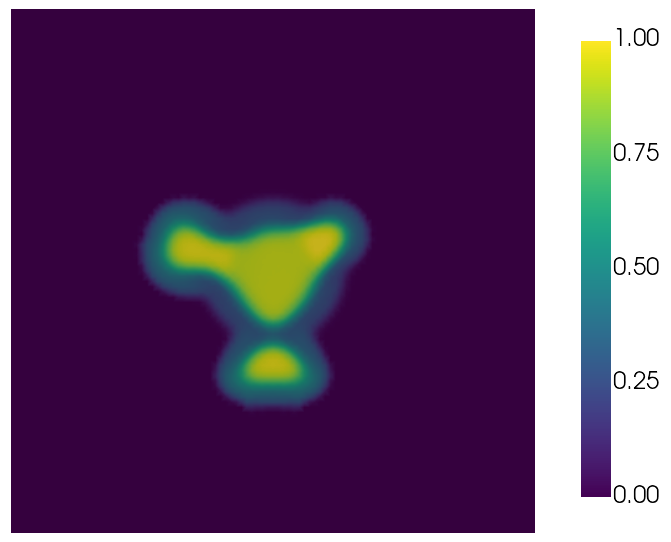}} &
		\raisebox{-0.47\height}{\includegraphics[scale=0.204]{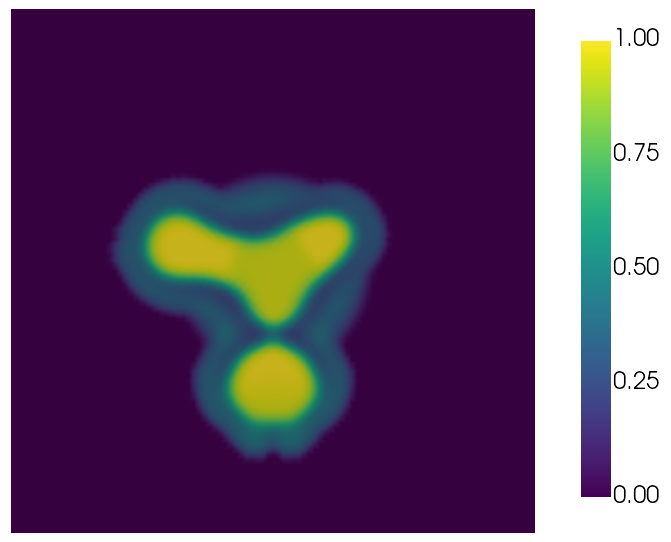}} &
		\raisebox{-0.47\height}{\includegraphics[scale=0.204]{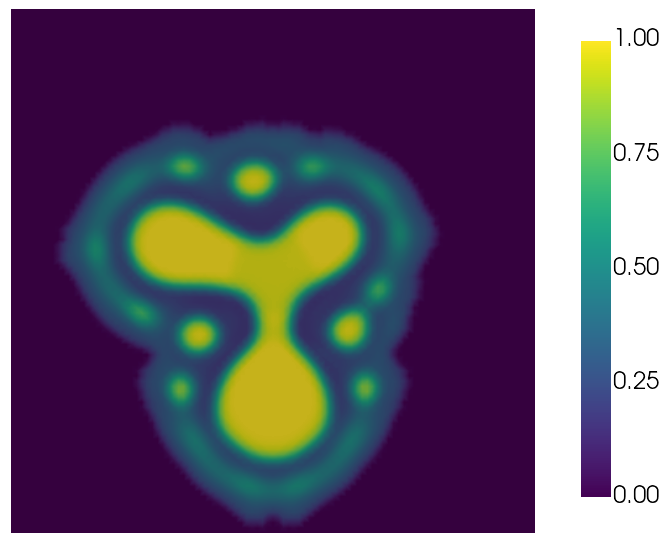}}\\
		\multirow{3}{*}{\vspace*{-2.7cm}$\boldsymbol{n}$}&\rotatebox[origin=c]{90}{\textbf{Symmetric \eqref{symmetric_functions}}} &
		\raisebox{-0.47\height}{\includegraphics[scale=0.204]{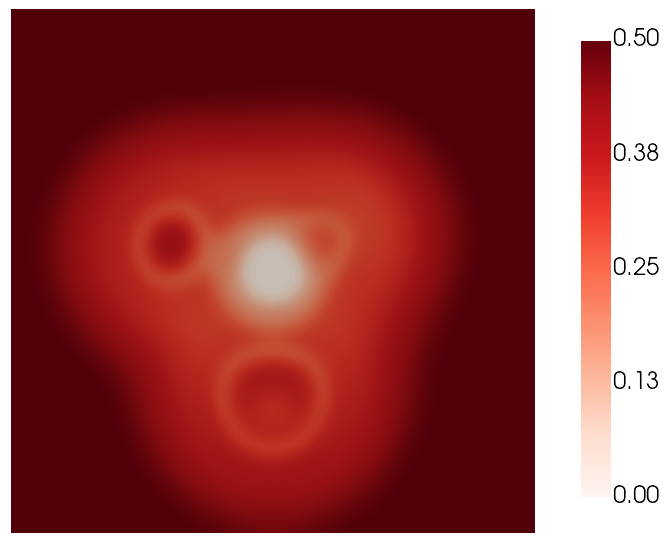}} &
		\raisebox{-0.47\height}{\includegraphics[scale=0.204]{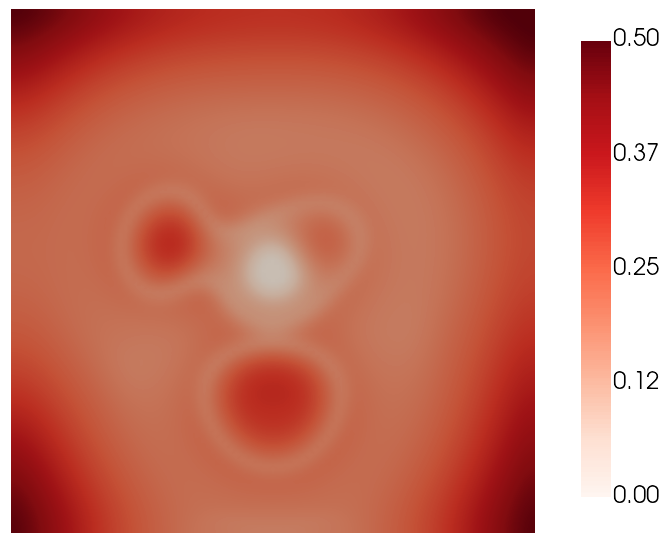}} &
		\raisebox{-0.47\height}{\includegraphics[scale=0.204]{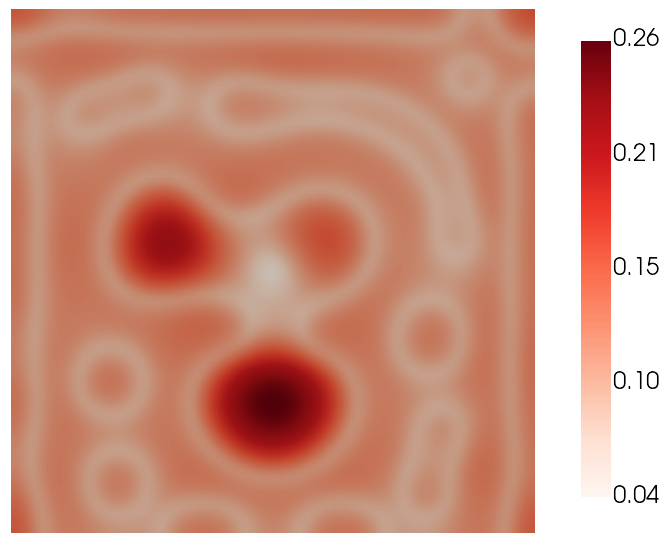}} \\
		&\rotatebox[origin=c]{90}{\textbf{Non-symmetric \eqref{nonsymmetric_functions}}} &
		\raisebox{-0.47\height}{\includegraphics[scale=0.204]{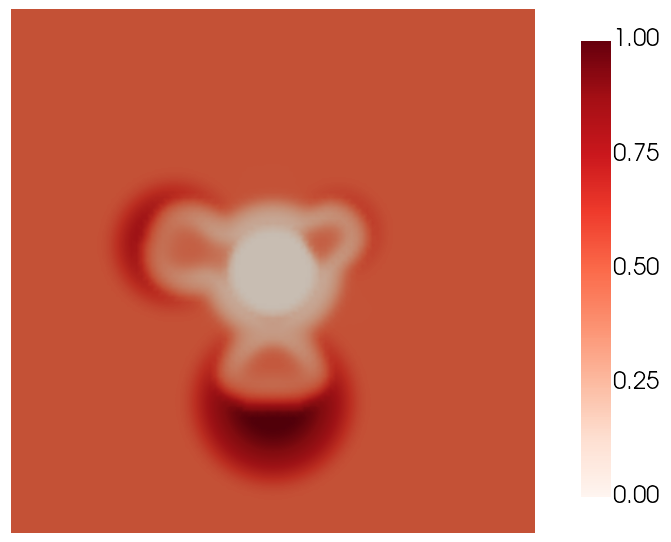}} &
		\raisebox{-0.47\height}{\includegraphics[scale=0.204]{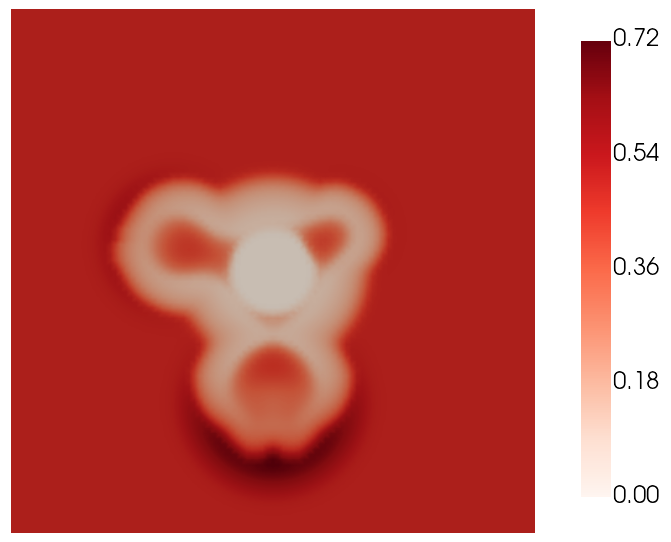}} &
		\raisebox{-0.47\height}{\includegraphics[scale=0.204]{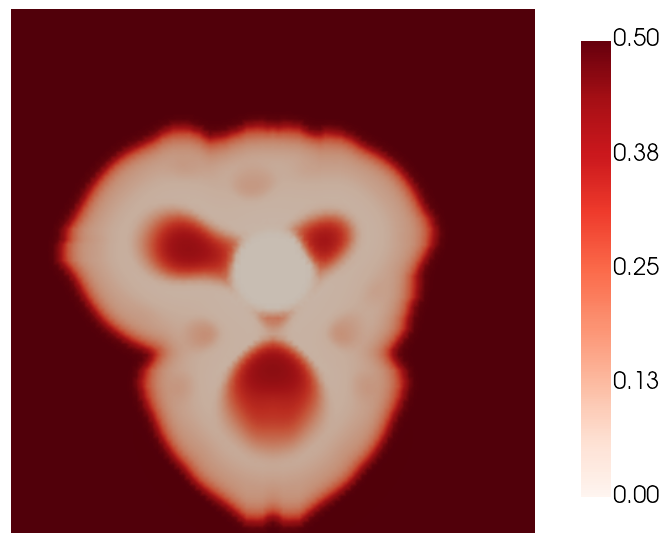}}
	\end{tabular}
	\caption{Tumor and nutrients for test \eqref{sec:numer-experiments_2} ($P_0=0.5$, $\chi_0=0.1$, $\Delta t=0.1$) at different time steps.}
	\label{fig:test-2_reference}
\end{figure}

\begin{figure}
	\centering
	\begin{tabular}{cccc}
			& \hspace*{-1cm}$t=30.0$ & \hspace*{-1cm}$t=50.0$ & \hspace*{-1cm}$t=100.0$ \\
			\rotatebox[origin=c]{90}{\textbf{Symmetric \eqref{symmetric_functions}}} &
			\raisebox{-0.47\height}{\includegraphics[scale=0.204]{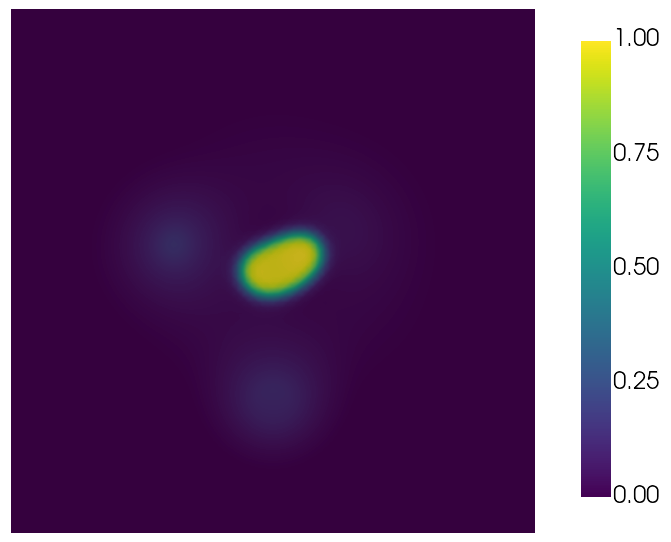}} &
			\raisebox{-0.47\height}{\includegraphics[scale=0.204]{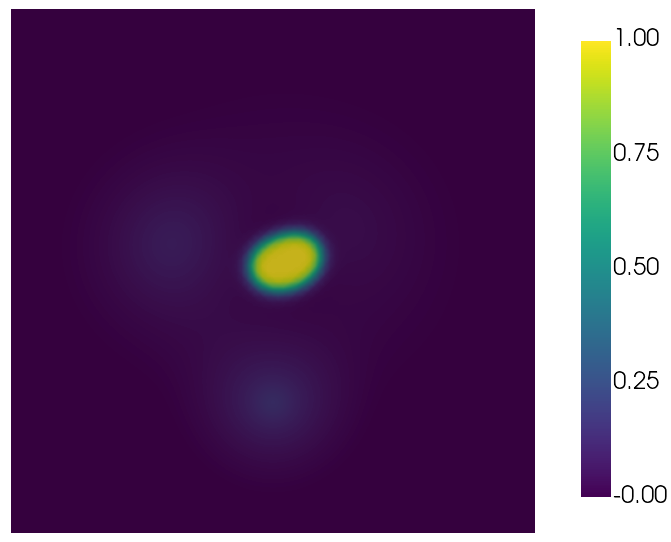}} &
			\raisebox{-0.47\height}{\includegraphics[scale=0.204]{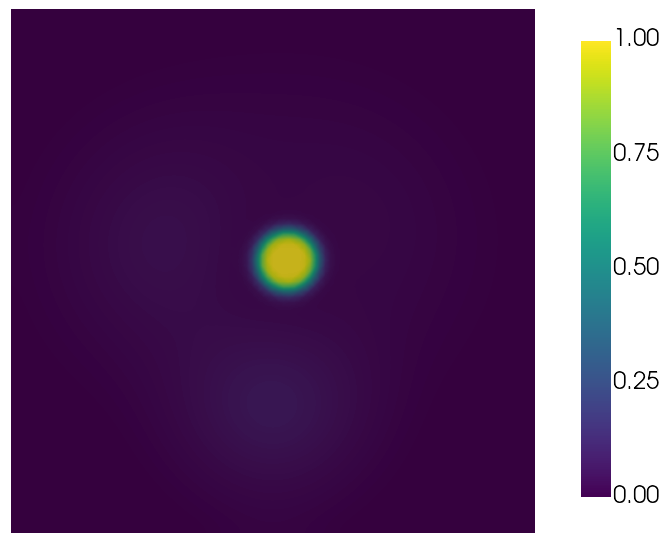}} \\
			\rotatebox[origin=c]{90}{\textbf{Non-symmetric \eqref{nonsymmetric_functions}}} &
			\raisebox{-0.47\height}{\includegraphics[scale=0.204]{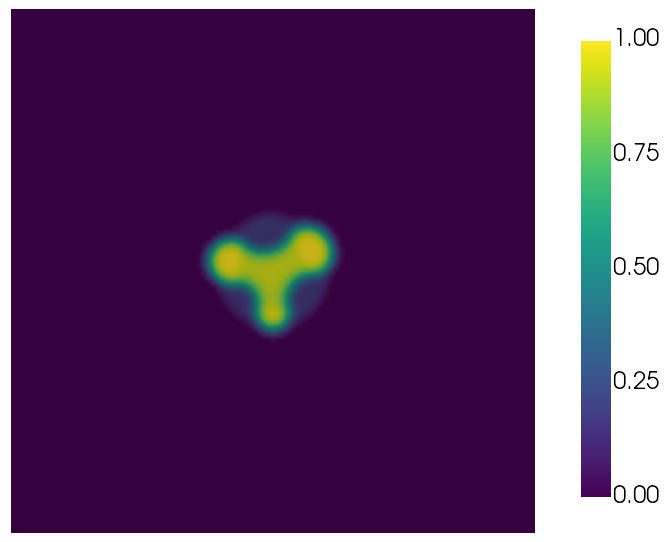}} &
			\raisebox{-0.47\height}{\includegraphics[scale=0.204]{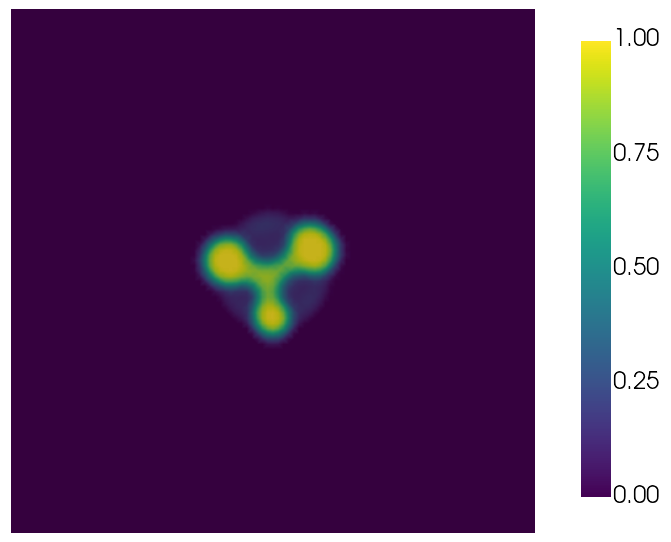}} &
			\raisebox{-0.47\height}{\includegraphics[scale=0.204]{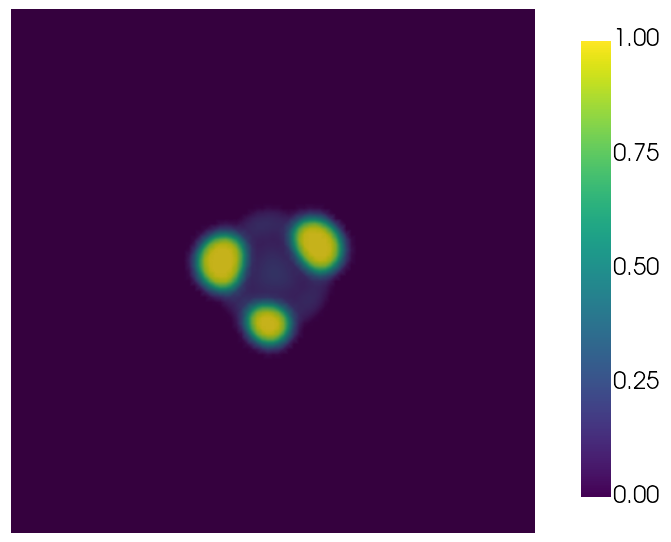}}
	\end{tabular}
	\caption{Tumor for test \eqref{sec:numer-experiments_2} ($P_0=0.001$, $\chi_0=0.1$, $\Delta t=0.1$) at different time steps.}
	\label{fig:test-2_P0-0.001_u}
\end{figure}

\begin{figure}
	\centering
	\begin{tabular}{cccc}
			& \hspace*{-1cm}$t=50.0$ & \hspace*{-1cm}$t=80.0$ & \hspace*{-1cm}$t=200.0$ \\
			\rotatebox[origin=c]{90}{\textbf{Symmetric \eqref{symmetric_functions}}} &
			\raisebox{-0.47\height}{\includegraphics[scale=0.204]{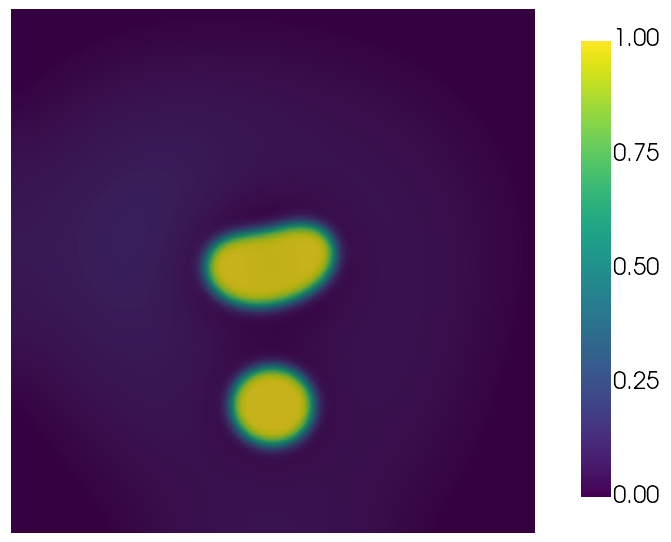}} &
			\raisebox{-0.47\height}{\includegraphics[scale=0.204]{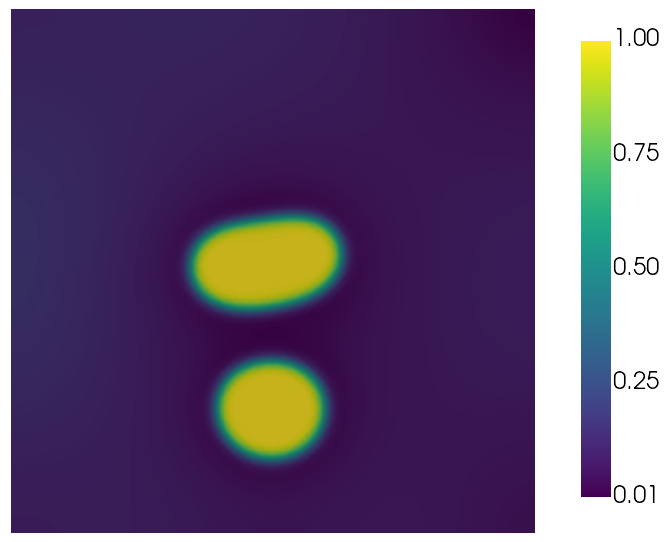}} &
			\raisebox{-0.47\height}{\includegraphics[scale=0.204]{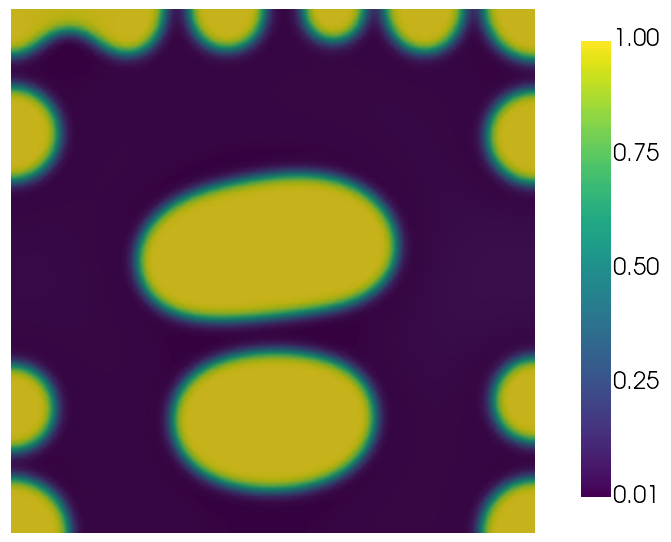}} \\
			\rotatebox[origin=c]{90}{\textbf{Non-symmetric \eqref{nonsymmetric_functions}}} &
			\raisebox{-0.47\height}{\includegraphics[scale=0.204]{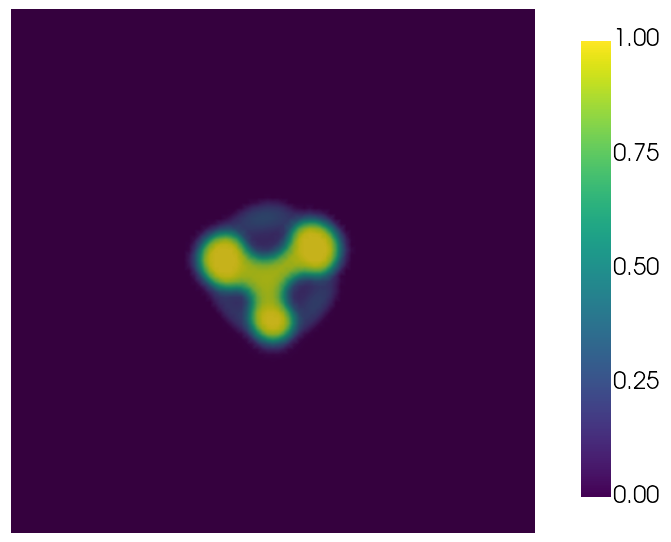}} &
			\raisebox{-0.47\height}{\includegraphics[scale=0.204]{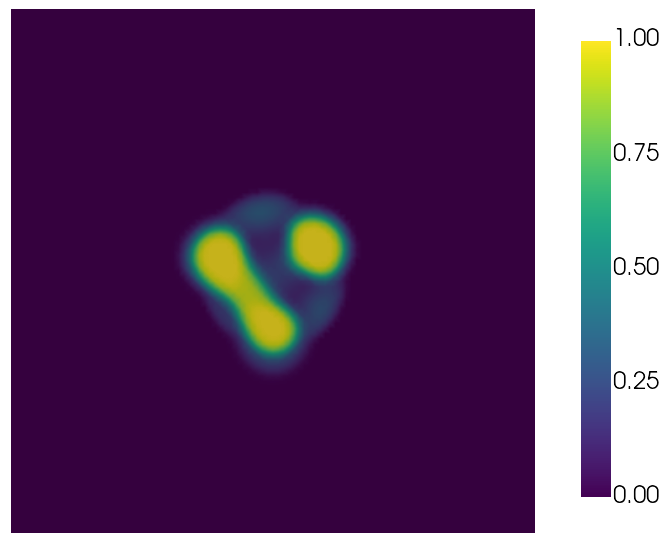}} &
			\raisebox{-0.47\height}{\includegraphics[scale=0.204]{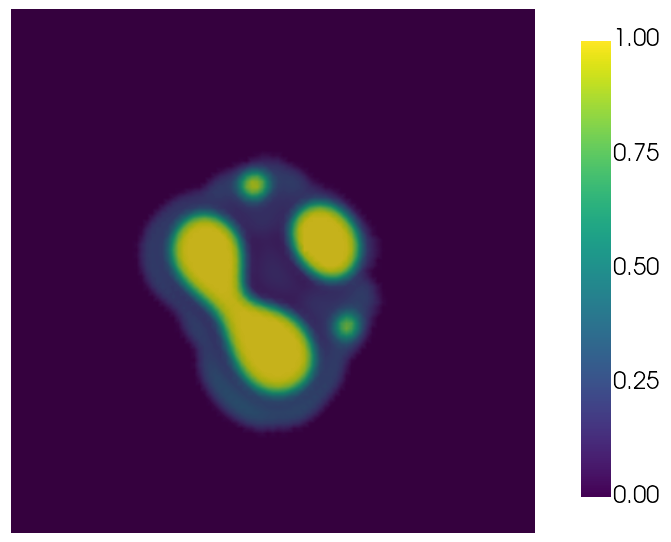}}
	\end{tabular}
	\caption{Tumor for test \eqref{sec:numer-experiments_2} ($P_0=0.05$, $\chi_0=0.1$, $\Delta t=0.1$) at different time steps.}
	\label{fig:test-2_P0-0.05_u}
\end{figure}

\begin{figure}
	\centering
	\begin{tabular}{cccc}
			& \hspace*{-1cm}$t=1.25$ & \hspace*{-1cm}$t=6.25$ & \hspace*{-1cm}$t=12.5$ \\
			\rotatebox[origin=c]{90}{\textbf{Symmetric \eqref{symmetric_functions}}} &
			\raisebox{-0.47\height}{\includegraphics[scale=0.204]{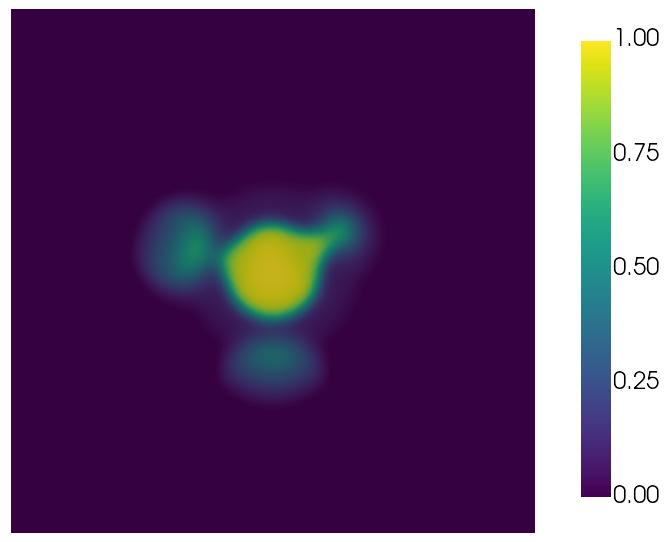}} &
			\raisebox{-0.47\height}{\includegraphics[scale=0.204]{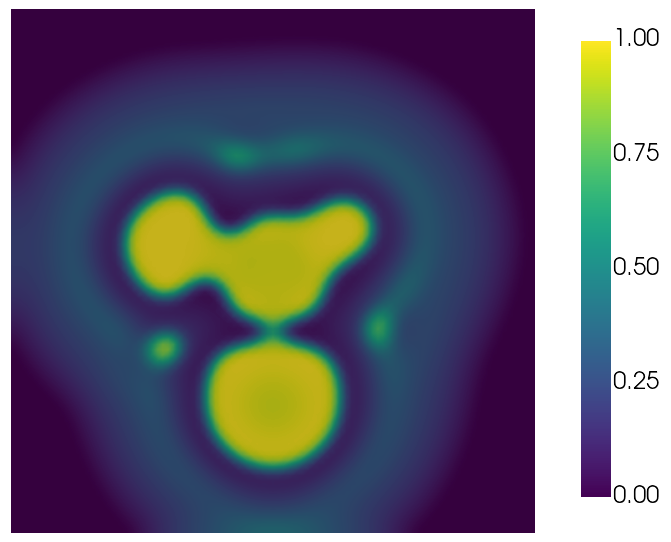}} &
			\raisebox{-0.47\height}{\includegraphics[scale=0.204]{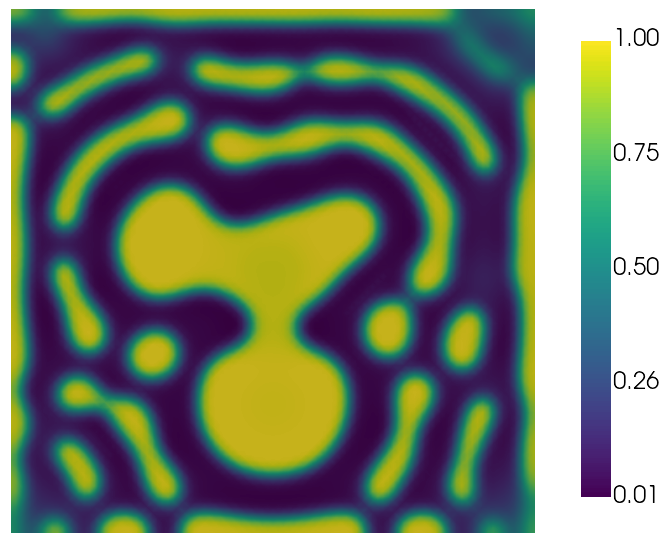}} \\
			\rotatebox[origin=c]{90}{\textbf{Non-symmetric \eqref{nonsymmetric_functions}}} &
			\raisebox{-0.47\height}{\includegraphics[scale=0.204]{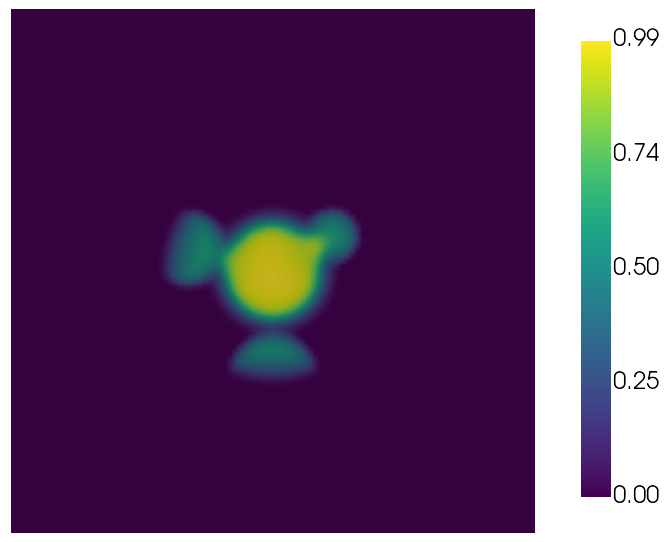}} &
			\raisebox{-0.47\height}{\includegraphics[scale=0.204]{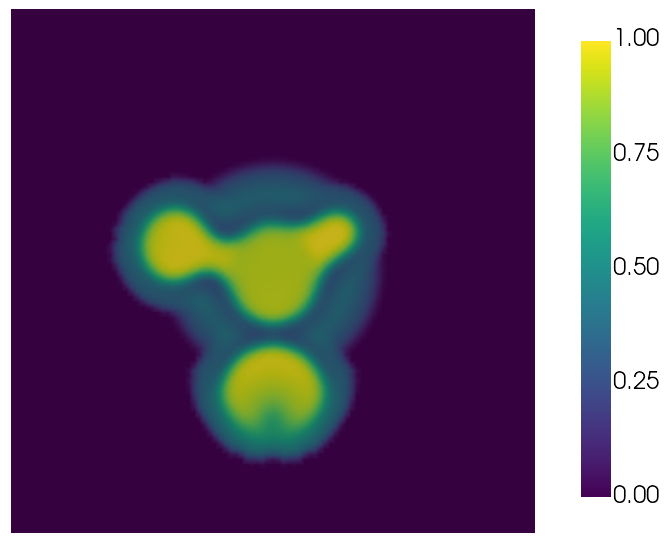}} &
			\raisebox{-0.47\height}{\includegraphics[scale=0.204]{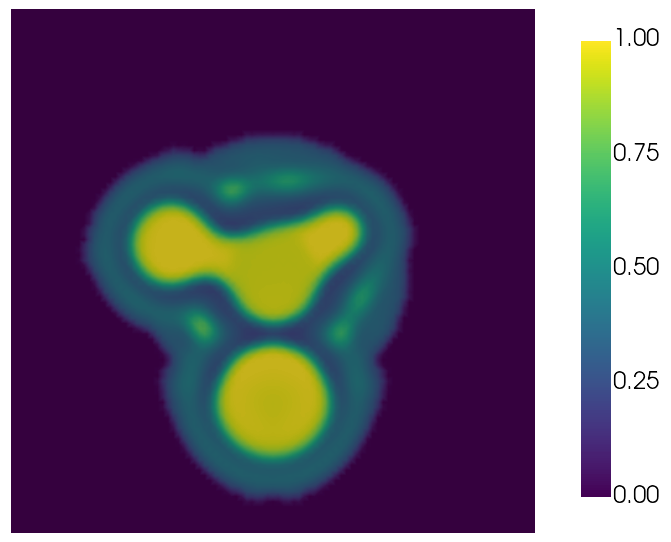}}
	\end{tabular}
	\caption{Tumor for test \eqref{sec:numer-experiments_2} ($P_0=2$, $\chi_0=0.1$, $\Delta t=0.025$) at different time steps.}
	\label{fig:test-2_P0-2_u}
\end{figure}

\begin{figure}
	\centering
	\begin{tabular}{cccc}
			& \hspace*{-1cm}$t=10$ & \hspace*{-1cm}$t=20$ & \hspace*{-1cm}$t=50$ \\
			\rotatebox[origin=c]{90}{\textbf{Symmetric \eqref{symmetric_functions}}} &
			\raisebox{-0.47\height}{\includegraphics[scale=0.204]{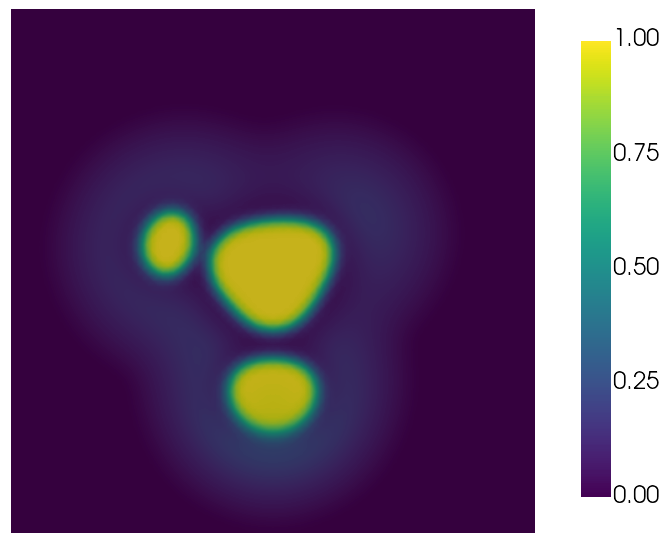}} &
			\raisebox{-0.47\height}{\includegraphics[scale=0.204]{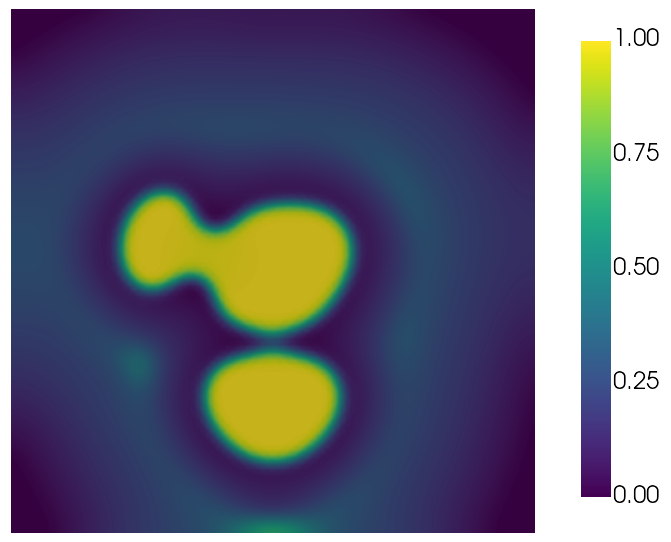}} &
			\raisebox{-0.47\height}{\includegraphics[scale=0.204]{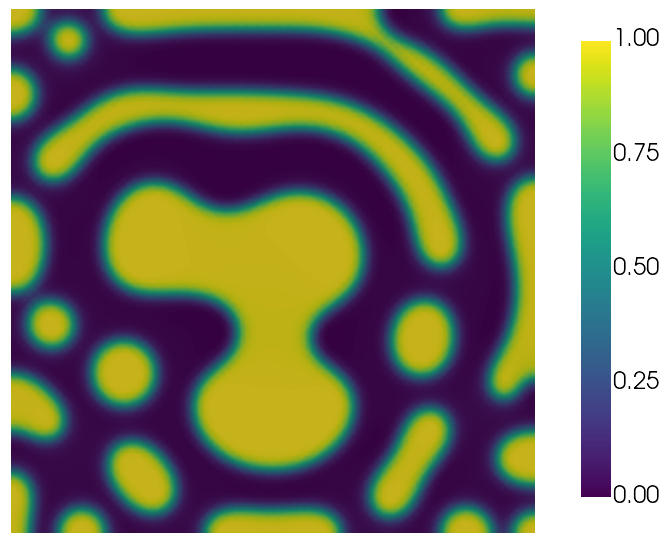}} \\
			\rotatebox[origin=c]{90}{\textbf{Non-symmetric \eqref{nonsymmetric_functions}}} &
			\raisebox{-0.47\height}{\includegraphics[scale=0.204]{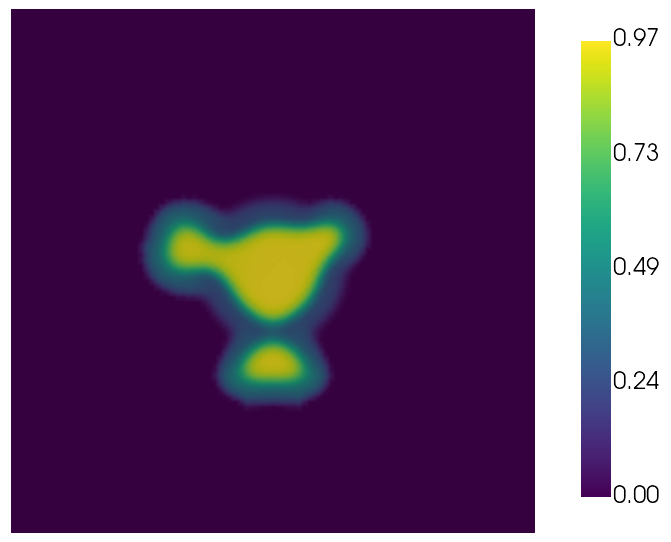}} &
			\raisebox{-0.47\height}{\includegraphics[scale=0.204]{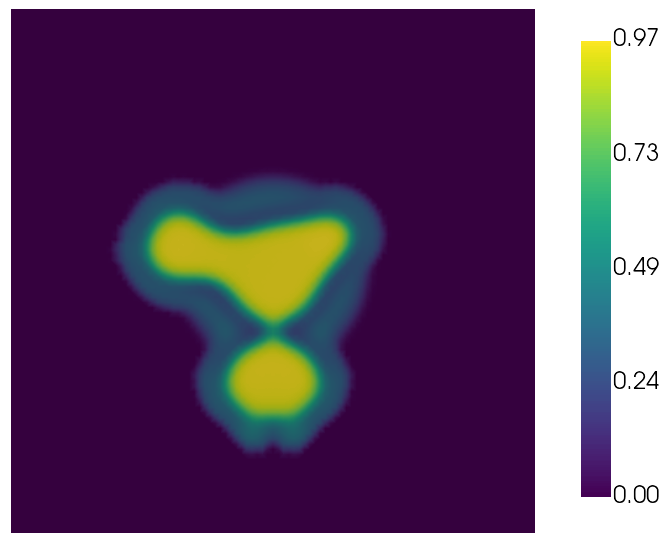}} &
			\raisebox{-0.47\height}{\includegraphics[scale=0.204]{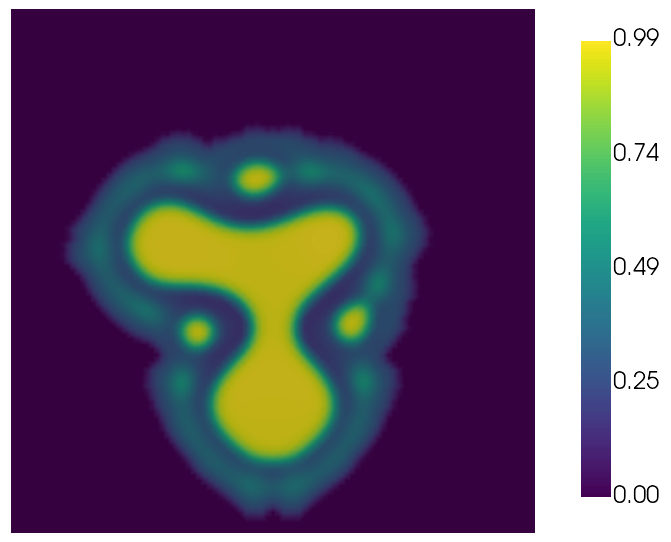}}
	\end{tabular}
	\caption{Tumor for test \eqref{sec:numer-experiments_2} ($P_0=0.5$, $\chi_0=0.01$, $\Delta t=0.1$) at different time steps.}
	\label{fig:test-2_chi-0.01_u}
\end{figure}

\begin{figure}
	\centering
	\begin{tabular}{cccc}
			& \hspace*{-1cm}$t=3$ & \hspace*{-1cm}$t=10$ & \hspace*{-1cm}$t=17$ \\
			\rotatebox[origin=c]{90}{\textbf{Symmetric \eqref{symmetric_functions}}} &
			\raisebox{-0.47\height}{\includegraphics[scale=0.204]{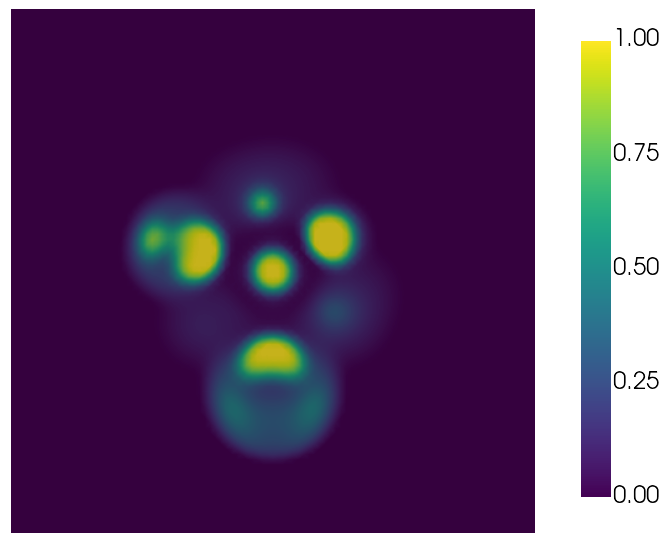}} &
			\raisebox{-0.47\height}{\includegraphics[scale=0.204]{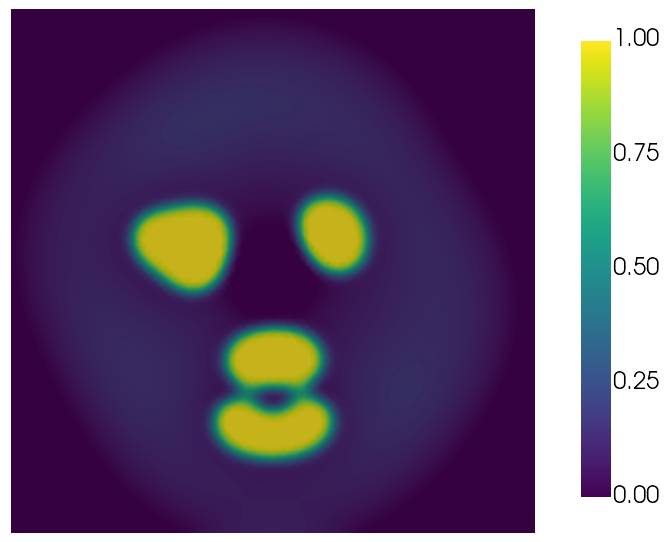}} &
			\raisebox{-0.47\height}{\includegraphics[scale=0.204]{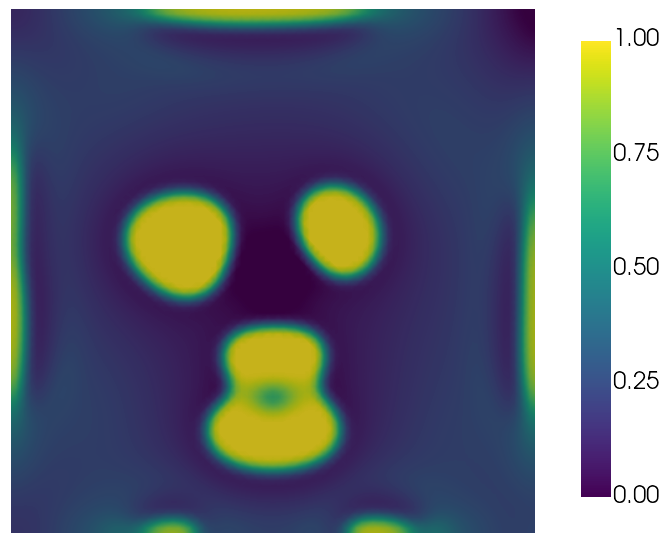}} \\
			\rotatebox[origin=c]{90}{\textbf{Non-symmetric \eqref{nonsymmetric_functions}}} &
			\raisebox{-0.47\height}{\includegraphics[scale=0.204]{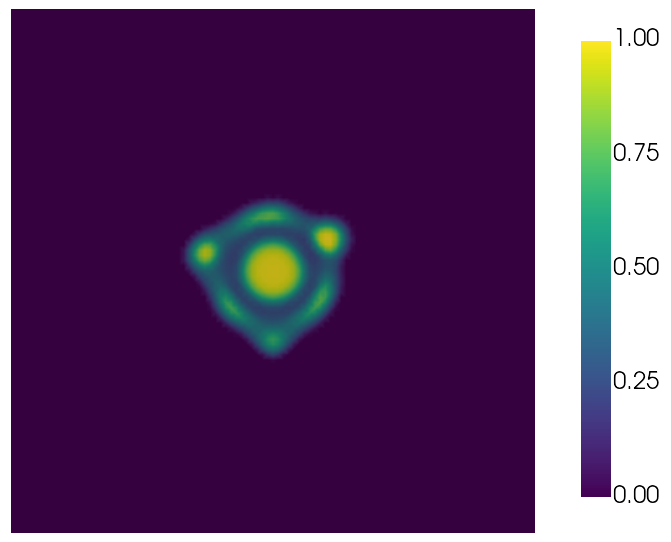}} &
			\raisebox{-0.47\height}{\includegraphics[scale=0.204]{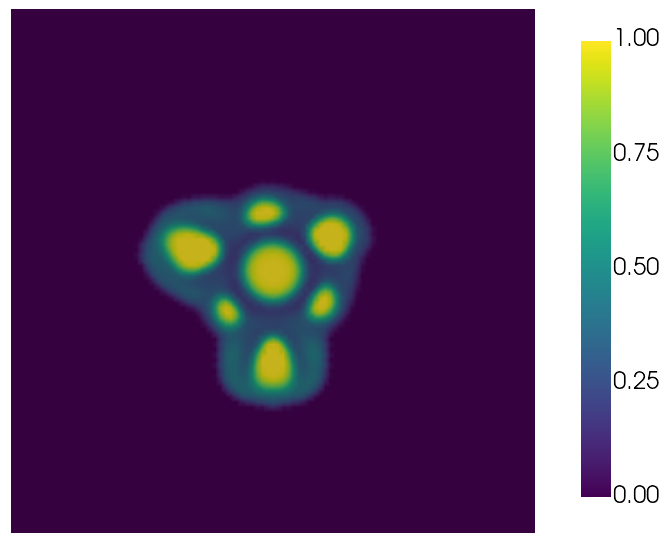}} &
			\raisebox{-0.47\height}{\includegraphics[scale=0.204]{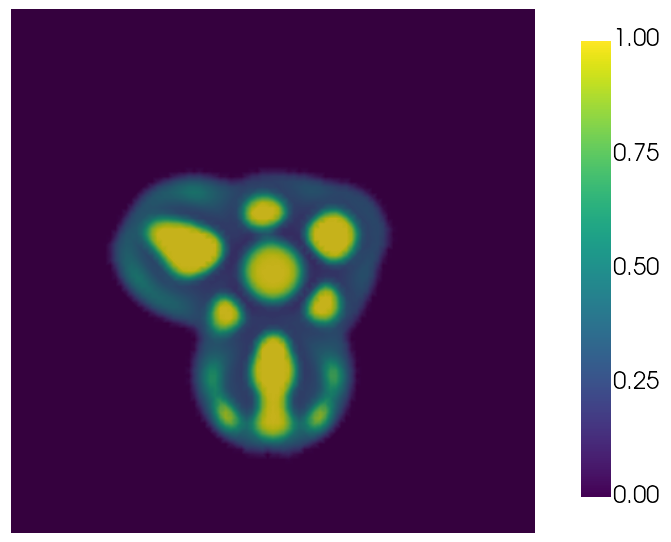}}
	\end{tabular}
	\caption{Tumor for test \eqref{sec:numer-experiments_2} ($P_0=0.5$, $\chi_0=0.5$, $\Delta t=0.01$) at different time steps.}
	\label{fig:test-2_chi-0.5_u}
\end{figure}

\begin{figure}
	\centering
	\begin{tabular}{cccc}
			& \hspace*{-1cm}$t=2.5$ & \hspace*{-1cm}$t=5$ & \hspace*{-1cm}$t=10$ \\
			\rotatebox[origin=c]{90}{\textbf{Symmetric \eqref{symmetric_functions}}} &
			\raisebox{-0.47\height}{\includegraphics[scale=0.204]{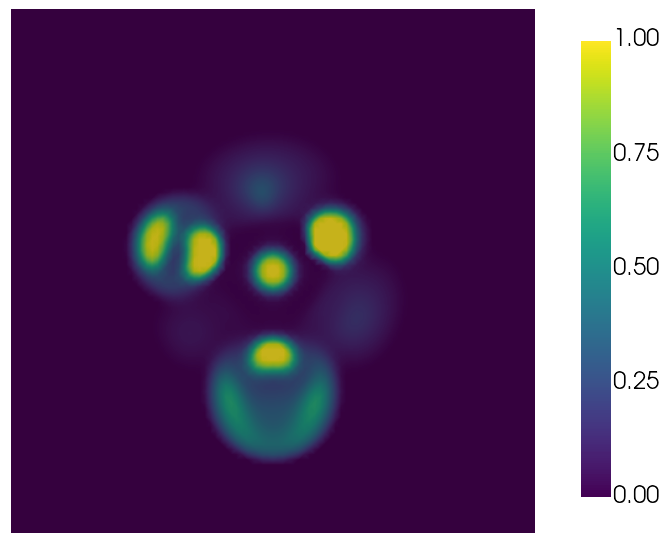}} &
			\raisebox{-0.47\height}{\includegraphics[scale=0.204]{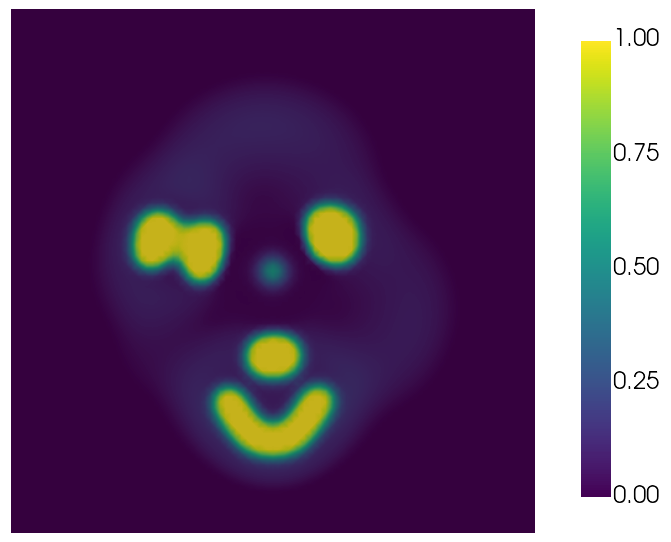}} &
			\raisebox{-0.47\height}{\includegraphics[scale=0.204]{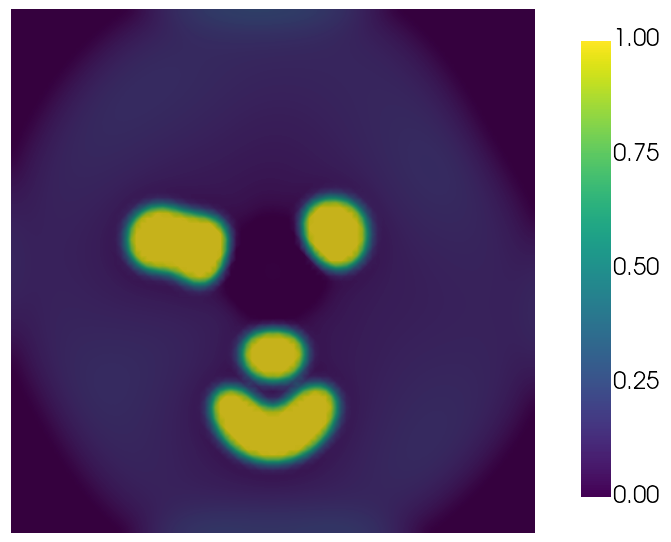}} \\
			\rotatebox[origin=c]{90}{\textbf{Non-symmetric \eqref{nonsymmetric_functions}}} &
			\raisebox{-0.47\height}{\includegraphics[scale=0.204]{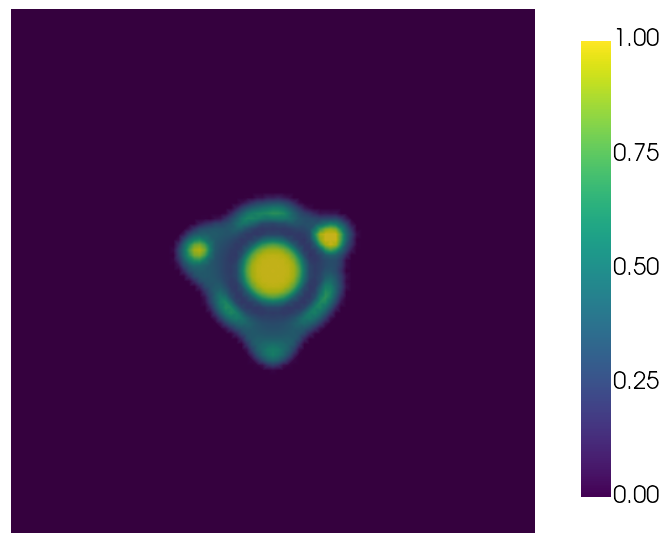}} &
			\raisebox{-0.47\height}{\includegraphics[scale=0.204]{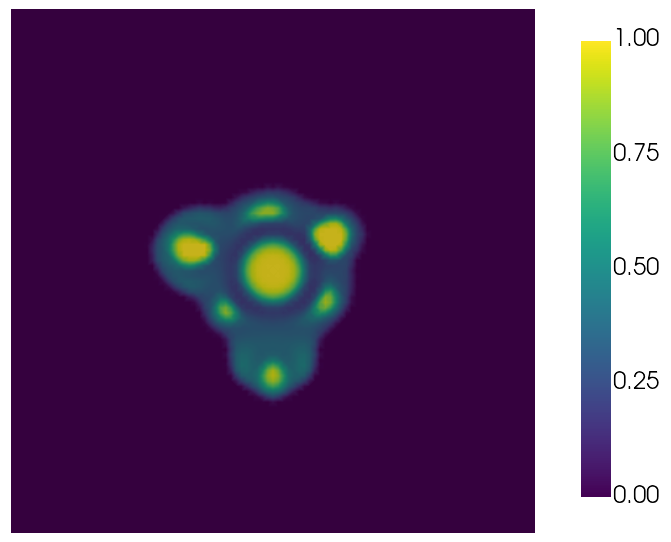}} &
			\raisebox{-0.47\height}{\includegraphics[scale=0.204]{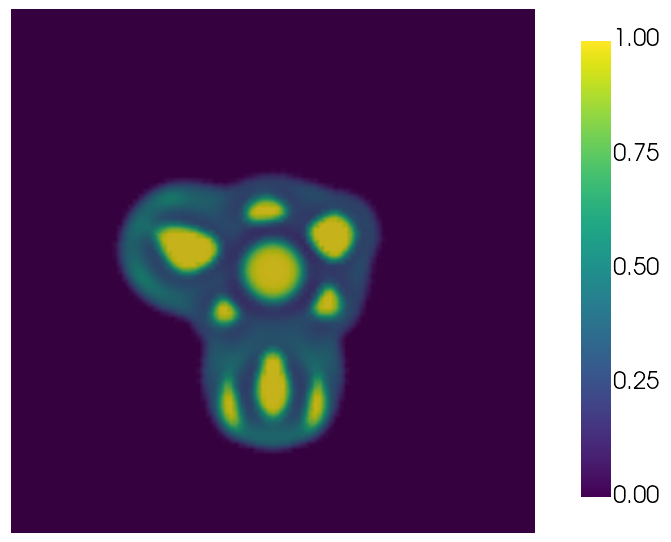}}
	\end{tabular}
	\caption{Tumor for test \eqref{sec:numer-experiments_2} ($P_0=0.5$, $\chi_0=1$, $\Delta t=0.01$) at different time steps.}
	\label{fig:test-2_chi-1_u}
\end{figure}

\section*{Acknowledgments}
The first author has been supported by \textit{UCA FPU contract UCA/REC14VPCT/2020 funded by Universidad de Cádiz} and by a \textit{Graduate Scholarship funded by the University of Tennessee at Chattanooga}. The second and third authors have been supported by \textit{Grant US-4931381261 (US/JUNTA/FEDER, UE)}.

\bibliography{biblio_database}

\end{document}